%% file: main.tex
\documentclass{article} 

\usepackage{graphicx}  

\usepackage{amsmath, mathrsfs, amssymb, amsthm, mathtools, bbm, bm} 

\usepackage{caption} 
\usepackage{subcaption} 

\usepackage{makecell} 
\usepackage{float} 

\usepackage{booktabs} 

\usepackage[english]{babel} 
\usepackage[utf8]{inputenc} 
\usepackage[T1]{fontenc} 
\usepackage[a4paper,margin=0.75in]{geometry}

\usepackage[notref,notcite]{showkeys}

\usepackage{tikz} 
\usetikzlibrary[topaths] 

\usepackage{footmisc} 

\usepackage{enumitem}

\usepackage{algorithm}
\usepackage{algcompatible} 

\newtheorem{theorem}{Theorem}
 
\newtheorem{definition}{Definition}
\newtheorem{lemma}{Lemma}
\newtheorem{remark}{Remark}

\newtheorem{proposition}{Proposition}
\newtheorem{assumption}{Assumption}

\newtheorem*{theorem*}{Theorem}
\newtheorem*{example*}{Example} 
\newtheorem*{definition*}{Definition}
\newtheorem*{lemma*}{Lemma}
\newtheorem*{remark*}{Remark}
\newtheorem*{corollary*}{Corollary}
\newtheorem*{proposition*}{Proposition}
\newtheorem*{assumption*}{Assumption}
\newtheorem*{claim*}{Claim}

\newtheoremstyle{TheoremNum}
        {\topsep}{\topsep}              
        {\itshape}                      
        {}                              
        {\bfseries}                     
        {.}                             
        { }                             
        {\thmname{#1}\thmnote{ \bfseries #3}}
\theoremstyle{TheoremNum}

\newtheoremstyle{LemmaNum}
        {\topsep}{\topsep}              
        {\itshape}                      
        {}                              
        {\bfseries}                     
        {.}                             
        { }                             
        {\thmname{#1}\thmnote{ \bfseries #3}}
\theoremstyle{LemmaNum}

\usepackage[hidelinks]{hyperref}

\input{macros}

\usepackage{natbib} 

\def\proofmark{0}

\begin{document} 

\title{Zeroth-order Stochastic Cubic Newton Method Revisited} 

\author{Yu Liu$^{a,}$\footnote{\textit{Email}: 22110840006@m.fudan.edu.cn; \textit{Address}: Songhu Rd. 2005, 200438, Yangpu District, Shanghai, China.} \quad Weibin Peng$^{b,}$\footnote{\textit{Email}: 23210180060@m.fudan.edu.cn; \textit{Address}: Handan Rd. 220, 200433, Yangpu District, Shanghai, China} \quad Tianyu Wang$^{a,}$\footnote{Corresponding author. \textit{Email}: wangtianyu@fudan.edu.cn; \textit{Address}: Songhu Rd. 2005, 200438, Yangpu District, Shanghai, China.} \quad Jiajia Yu$^{c,}$\footnote{\textit{Email}: jiajia.yu@duke.edu; \textit{Address}: 120 Science Drive, 222 Physics Building, Durham, NC, 27708, USA} \\
\small $^{a}$Shanghai Center for Mathematical Sciences, Fudan University, Shanghai, China \\ 
\small $^{b}$School of Mathematical Sciences, Fudan University, Shanghai, China \\ 
\small $^{c}$Department of Mathematics, Duke University, Durham NC, USA
} 
\date{} 
\maketitle

\begin{abstract}
This paper studies stochastic minimization of a finite-sum loss $ F (\mathbf{x}) = \frac{1}{N} \sum_{\xi=1}^N f(\mathbf{x};\xi) $. In many real-world scenarios, the Hessian matrix of such objectives exhibits a low-rank structure on a batch of data. At the same time, zeroth-order optimization has gained prominence in important applications such as fine-tuning large language models. Drawing on these observations, we propose a novel stochastic zeroth-order cubic Newton method that  leverages the low-rank Hessian structure via a matrix recovery-based estimation technique. 
We establish that for important problems in $\mathbb{R}^n$, $\mathcal{O}\left(\frac{n}{\eta^{\frac{7}{2}}}\right)+\widetilde{\mathcal{O}}\left(\frac{n^2 }{\eta^{\frac{5}{2}}}\right)$ function evaluations suffice to attain a second-order $\eta$-stationary point with high probability. This represents an improvement in dimensional dependence over existing methods. 
\end{abstract}

\input{tex/new_intro_newton}

\input{tex/algo-brief}

\input{tex/algo}

\input{tex/experiment}

\newpage 

\bibliographystyle{apalike} 
\bibliography{references} 

\appendix

\input{tex/appendix}

\end{document}

%% file: macros.tex
\renewcommand{\Pr}{ \mathbb{P} }

\newcommand{\x}{ \mathbf{x} }
\newcommand{\y}{ \mathbf{y} }
\newcommand{\z}{ \mathbf{z} }

\renewcommand{\v}{ \mathbf{v} } 
\renewcommand{\u}{ \mathbf{u} }

\newcommand{\g}{ \mathbf{g} } 
\newcommand{\e}{ \mathbf{e} }

\newcommand{\R}{\mathbb{R}}

\renewcommand{\S}{\mathcal{S}}
\newcommand{\E}{\mathbb{E}}

\renewcommand{\P}{\mathcal{P}}

\newcommand{\tr}{\mathrm{tr}}

\renewcommand{\[}{\left[ }
\renewcommand{\]}{\right] }

\newcommand{\<}{\left< }
\renewcommand{\>}{\right> }

\renewcommand{\(}{\left( }
\renewcommand{\)}{\right) }

\newcommand{\wh}{\widehat }

%% file: tex/new_intro_newton.tex
\section{Introduction}
\label{sec:intro}
Finite-sum loss minimization is an important task in stochastic optimization and machine learning, where the goal is to minimize the sum of individual losses over all data points: 
\begin{align}
    \min_{\x \in \R^n} F (\x) := \frac{1}{N} \sum_{\xi=1}^N f(\x;\xi)
    \label{eq:obj} 
\end{align} 
where each $ f(\cdot;\xi) $ represents the loss incurred by a data point. 


\subsection{Intrinsic Low-rank Hessians on a Batch of Data} 
\label{sec:low-rank-motivation}

The finite-sum objective's Hessian $\nabla^2 f(\x;\xi)$ is often low-rank on a data batch. We highlight key examples below.

\textbf{(M1)} \textbf{Linear regression} models relationships between features and targets by minimizing the mean squared error (MSE): 
\begin{align}
     F (\x) := \frac{1}{N} \sum_{\xi =1}^N f (\x; \xi) , 
    \qquad 
    f (\x; \xi) = \left(b_\xi- \mathbf{z}_{\xi}^\top \x \right)^2,
    \label{eq:obj-linear}
\end{align}
where $(\mathbf{z}_{\xi},b_\xi)$ represent the feature-label pair.

The Hessian for \textbf{linear regression} is:
\begin{align*}
\nabla^2 f(\x; \xi) = 2\z_\xi \z_\xi^\top \quad \forall \x, \forall \xi, 
\end{align*}
which is clearly low-rank. 

\textbf{(M2)} The classic \textbf{logistic regression} is a binary classification model optimized using logistic loss: 
\begin{align} 
    F (\x) := \frac{1}{N} \sum_{\xi =1}^N f (\x; \xi) , 
    \qquad 
    f (\x; \xi) = \log\left(1+e^{-y_\xi \z_{\xi}^\top \x }\right), 
    \label{eq:obj-logi} 
\end{align} 
where $ \( \z_{\xi}, y_\xi \) $  denote the feature and label of the $\xi$-th training sample, with $y_{\xi} \in \{-1,+1\}$.  

The Hessian for \textbf{logistic regression} satisfies: 
    \begin{align*} 
        \nabla^2 f (\x; \xi) = \mathtt{sig} (y_\xi\z_\xi^\top \x ) ( 1 - \mathtt{sig}(y_\xi\z_\xi^\top \x ) ) \z_\xi \z_\xi^\top , \quad \forall \x, \forall \xi , 
    \end{align*} 
    where $\mathtt{sig} (\cdot) $ is the sigmoid function. For any $\x$, the loss $f(\x; \xi)$ admits a low-rank Hessian.  

\textbf{(M3)} \textbf{Neural Networks}: The objective for training a neural network using least square loss can be written as
\begin{align} 
    F (W) \!:=\! \frac{1}{N} \sum_{\xi =1}^N f (W; \xi) , 
    \;
    f (W; \xi) \!= \!\| W_L \sigma_L (W_{L-1} \cdots \sigma_2 (W_2 \sigma_1 (W_1 \z_\xi))) \!-\! y_\xi \|^2 , 
    \label{eq:obj-nn} 
\end{align} 
where $\sigma_i$ is the activation function (typically ReLU \citep{nair2010rectified}) of the $i$-th hidden layer, $W$ is the weight tensor with $W_i \in \R^{d_{i} \times d_{i-1}}$, and $\( \z_{\xi}, y_\xi \)$ denote the input and output of the $\xi$-th  trainning data, respectively. 

The Hessian of such objective is also low-rank, as demonstrated in \citep{singh2021analytic}; We also provide a derivation for this in the Appendix for completeness.

\subsection{Stochastic Zeroth-order Optimization in Modern Machine Learning} 
\label{sec:zeroth-motivation}

Real-world objectives often lack closed-form solutions, involve complexity, or require privacy—making zeroth-order optimization essential when only function evaluations are available.

Also, recent advances show zeroth-order methods excel in fine-tuning LLMs\citep{malladi2023finetuning,chen2023deepzero,liu2024sparse}. Unlike backpropagation, which requires storing gradients, zeroth-order methods estimate gradients via parameter perturbations, drastically reducing memory usage. Techniques like the ``random seed trick'' \citep{malladi2023finetuning} further enhance efficiency. By relying solely on function evaluations, these methods offer a scalable, hardware-efficient approach for large-scale ML.

\subsection{The Key Question}

At this point, a compelling question emerges: 
\begin{quote} 
\textit{Given the ubiquity of low-rank Hessian structures and the increasing significance of stochastic zeroth-order optimization in modern machine learning, can we exploit low-rank Hessian structures to improve the efficiency of stochastic zeroth-order optimization algorithms? \!\textbf{(Q)}} 
\end{quote} 

To our knowledge, no existing work directly targets at this important question. In this paper, we provide a positive answer to this question.

\begin{figure}[H]
    \centering
    \input{./tex/new_graph_motivation.tex}
    \caption{Summary of the motivation of the key question.}
    \label{fig:motivation}
\end{figure}

\subsection{Our Approach}





Now we outline the approach we use to address Question \textbf{(Q)}. 


\subsubsection{Low-rank Hessian Estimation} 
\label{sec:low-rank-hess-intro}


The Hessian matrix frequently exhibits low-rank structure when evaluated on a batch of data—a prevalent characteristic in modern applications (Section \ref{sec:low-rank-motivation}). This motivates our focus on low-rank Hessian estimation. To reconstruct such an $n \times n$ matrix $H$ using $\ll n^2$ finite-difference operations, we require efficient matrix measurements of $H$. Each matrix measurement involves performing a matrix multiplication or Frobenius inner product between the hidden matrix $H$ and some chosen measurement matrices $A$'s. 
With these matrix measurements in place, 
we adopt the trace norm minimization framework \citep{fazel2002matrix,doi:10.1137/070697835,candes2010power,gross2011recovering,candes2012exact}: 
\begin{align} 
    \min_{\wh{H} \in \R^{n \times n} } \| \wh{H} \|_1 , 
    \quad 
    \text{ subject to } 
    \quad 
    \S \wh{H} = \S H , \label{eq:for-hess-est} 
\end{align} 
where $\S$ is defined as $\frac{1}{M} \sum_{i=1}^M \P_i$, with each $\P_i$ representing a matrix measurement operation that can be obtained via $\mathcal{O} (1)$ finite-difference computations. For our problem, we highlight two critical properties of the operator $\P_i$ employed in this work:
\begin{itemize} 
    \item \textbf{(R1)} The operator $\P_i$ differs fundamentally from the measurement operators used in matrix completion. This distinction avoids the need for the incoherence assumption typically imposed in matrix completion settings.
    \item \textbf{(R2)} The operator $\mathcal{P}_i$
    cannot require generic inner products $\langle H, A\rangle$ unless expressible as simple quadratic forms. This constraint arises because such inner products cannot be computed efficiently via finite-difference methods. 
    
\end{itemize} 

These constraints render most of existing matrix recovery frameworks inadequate: 
\begin{itemize} 
    \item \textbf{Matrix completion methods} \citep{fazel2002matrix,cai2010singular,candes2010matrixnoise,5466511,fornasier2011low,gross2011recovering,recht2011simpler,chen2015incoherence} violate \textbf{(R1)} by requiring unrealistic incoherence assumptions---specifically, uniform nontrivial angles between the Hessian's singular spaces and the observation basis at every $\x$ during optimization.
    
    \item \textbf{Matrix regression methods}  \citep{tan2011rank,Chandrasekaran2012,10.1214/10-AOS860,10.1214/20-AOS1980} violate \textbf{(R2)}.
    They rely on linear measurements $\langle H, A_i\rangle = \tr(H^* A_i)$ for general measurement matrices $\{A_i\}$. In zeroth-order settings, such measurements are prohibitively expensive: Taylor expansions represent Hessian as a bilinear operator, making generic $\langle H, A_i\rangle$ evaluations require $\gg \mathcal{O}(1)$ operations. This conflicts with finite-difference efficiency unless $\langle H, A_i\rangle$ admits simple computation.


\end{itemize}

To address these limitations, we develop two novel estimators satisfying both \textbf{(R1)} and \textbf{(R2)}. Derived from distinct undersampled measurements via nuclear norm minimization, they enable exact rank-$r$ Hessian recovery with high probability using either $\mathcal{O}(nr^2 \log^2 n)$ or $\mathcal{O}(nr)$ finite-difference operations—without incoherence assumptions.


\subsubsection{Zeroth-order Stochastic Cubic Newton Method} 

Based on our low-rank Hessian estimation, we design zeroth-order stochastic cubic Newton method for smooth programs to solve the problem (\ref{eq:obj}). For real-valued functions over $\R^n$ with rank-$r$ Hessians, we establish the following guarantees:

\begin{itemize} 
    \item Using two distinct estimation techniques, we achieve an expected second-order $\eta$-stationary point (Definition \ref{def:second-order}) with total function evaluation complexities of either $\mathcal{O}\left(\frac{n}{\eta^{7/2}}\right)+\widetilde{\mathcal{O}}\left(\frac{n^2 r^2}{\eta^{5/2}}\right)$ or $ \mathcal{O}\left(\frac{n}{\eta^{7/2}}\right)+{\mathcal{O}}\left(\frac{n^2 r }{\eta^{5/2}}\right) $ as formalized in Theorem \ref{thm:epsilon local optima-0}.

\end{itemize} 


Our results offer a significant reduction in function evaluation complexity compared to the previous work of \cite{balasubramanian2021zeroth}, which has a function evaluation complexity of $ \mathcal{O}\left(\frac{n}{\eta^{7/2}}\right)+\widetilde{\mathcal{O}}\left(\frac{n^4}{\eta^{5/2}}\right) $ for a second order $\eta$-stationary point in the expectation sense. 
This reduction is particularly beneficial when dealing with the low-rank Hessian structure, which is ubiquitous as discussed in Section \ref{sec:low-rank-motivation}. 



\subsection{Prior Arts} 
\label{sec:existing-hess}

A significant portion of the optimization literature is dedicated to studying zeroth-order optimization (or derivative-free optimization) methods (e.g., \citep{conn2009introduction,shahriari2015taking}). 
Among many zeroth-order optimization mechanisms, a classic and prosperous line of works focuses on estimating gradient/Hessian using zeroth-order information and use the estimated gradient/Hessian for downstream optimization algorithms. 

In recent decades, due to lack of direct access to gradients in real-world applications, zeroth-order optimization has attracted the attention of many researchers \citep{flaxman2005online,ghadimi2013stochastic,duchi2015optimal,nesterov2017random,liu2018zeroth,chen2019zo,balasubramanian2021zeroth}. In particular, \cite{flaxman2005online} introduced the single-point gradient estimator for the purpose of bandit learning. Afterwards, many modern gradient/Hessian estimators have been introduced and subsequent zeroth-order optimization algorithms have been studied. To name a few, \cite{duchi2015optimal,nesterov2017random} have studied zeroth-order optimization algorithm for convex objective and established in expectation convergence rates. 
Replacing the exact gradient by a gradient estimator, many widely used first-order methods now have their zeroth-order counterparts, including ZO-SGD \citep{ghadimi2013stochastic}, ZO-SVRG \citep{liu2018zeroth} and ZO-adam \citep{chen2019zo}, just to name a few. 

A line of research utilizes the principles of compressed sensing to zeroth-order optimization \citep{Bandeira2012,CAI2022242,cai2022zeroth}. More specifically, \cite{CAI2022242} used one-bit compressed sensing for gradient estimation, and \cite{cai2022zeroth} used CoSaMP \citep{needell2009cosamp} to estimate compressible gradients, which leads to significantly reduced sample complexity for functions with compressible gradients. 





Regarding Hessian estimation, its foundations trace back to Newton-era finite-difference principles \citep{taylor-fd}, with modern contributions emerging from diverse fields including \citep{broyden1973local,fletcher2000practical,balasubramanian2021zeroth}.
Quasi-Newton methods (e.g., \citep{goldfarb1970family,shanno1970conditioning,broyden1973local,davidon1991variable}) have historically leveraged gradient-based Hessian estimators to enhance iterative optimization.
Recent advances integrate Stein’s identity \citep{10.1214/aos/1176345632}—a tool from statistical estimation—to design novel Hessian estimators. For instance, \cite{balasubramanian2021zeroth} introduced a Stein-type Hessian estimator and integrated it with the cubic regularized Newton’s method \citep{nesterov2006cubic} to address non-convex optimization challenges. Expanding further, \cite{li2023stochastic} generalized these estimators to Riemannian manifolds, broadening their applicability to geometric optimization. Concurrently with the advancements in \citep{balasubramanian2021zeroth,li2023stochastic}, the Hessian estimator analyzed in \cite{wang2022hess,10.1093/imaiai/iaad014} serves as the direct inspiration for the current work.

\subsection{Summary of Contributions}
\label{sec:contribution}
To sum up, this paper focuses on the important problem highlighted in Figure \ref{fig:motivation}, and provides a positive answer to Question \textbf{(Q)}. Our main contributions are summarized below. 
\begin{itemize}
     \item We introduce a new stochastic zeroth-order cubic Newton method that takes advantage of the low-rank structure of the Hessian on a batch of training data. This new version improves the sample complexity of the existing state-of-the-art zeroth-order stochastic cubic Newton method. 
    In particular, when the rank of the Hessian is bounded by $r$, we guarantee an expected second-order $\eta$-stationary point  using either $\mathcal{O}\left(\frac{n}{\eta^{7/2}}\right)+\widetilde{\mathcal{O}}\left(\frac{n^2 r^2}{\eta^{5/2}}\right)$ or $ \mathcal{O}\left(\frac{n}{\eta^{7/2}}\right)+{\mathcal{O}}\left(\frac{n^2 r }{\eta^{5/2}}\right) $ function evaluations. This represents a substantial improvement over the current state-of-the-art complexity of $\mathcal{O}\left(\frac{n}{\eta^{7/2}}\right) + \widetilde{\mathcal{O}}\left(\frac{n^4}{\eta^{5/2}}\right)$, particularly in high-dimensional problems where $r \ll n$.
\end{itemize}


\section{Notations and Conventions}
\label{sec:notation}

First of all, we agree on the following assumptions throughout the rest of the paper. 



\begin{assumption}
    The Hessian is Lipschitz continuous, i.e., there exists $L_2>0$ such that for any $\x,\y \in \R^n$ and any $\xi \in \Theta$, 
    \begin{align*} 
        \| \nabla^2 f (\x; \xi) - \nabla^2 f (\y; \xi) \| \le L_2 \| \x - \y \|.
    \end{align*} 
    where $\|\cdot\|$ denotes spectral norm for matrices and $2$-norm for vectors.
\label{as:H lip cts }
\end{assumption}

\begin{assumption} 
    \label{assumption:low-rank} 
    There exists $r \ll n$, such that for any $ \x \in \R^n $ and any $\xi \in \Theta$, $rank (\nabla^2 f (\x; \xi)) \le r$. 
     
\end{assumption}

Additionally, we use $C$ and 
$c$ to denote generic positive absolute constants, whose numerical values may differ between occurrences and remain unspecified.
We use $ K $'s ($K, K_1, K_2,$ etc.) to denote constants whose value will be held fixed throughout.

%% file: tex/new_graph_motivation.tex
\tikzset{every picture/.style={line width=0.75pt}} 

\begin{tikzpicture}[x=0.75pt,y=0.75pt,yscale=-1,xscale=1]

\draw   (209.33,29.2) .. controls (209.33,24.67) and (213,21) .. (217.53,21) -- (441.13,21) .. controls (445.66,21) and (449.33,24.67) .. (449.33,29.2) -- (449.33,53.8) .. controls (449.33,58.33) and (445.66,62) .. (441.13,62) -- (217.53,62) .. controls (213,62) and (209.33,58.33) .. (209.33,53.8) -- cycle ;
\draw   (16.33,123.53) .. controls (16.33,116.24) and (22.24,110.33) .. (29.53,110.33) -- (273.13,110.33) .. controls (280.42,110.33) and (286.33,116.24) .. (286.33,123.53) -- (286.33,163.13) .. controls (286.33,170.42) and (280.42,176.33) .. (273.13,176.33) -- (29.53,176.33) .. controls (22.24,176.33) and (16.33,170.42) .. (16.33,163.13) -- cycle ;
\draw   (375.33,123.53) .. controls (375.33,116.24) and (381.24,110.33) .. (388.53,110.33) -- (632.13,110.33) .. controls (639.42,110.33) and (645.33,116.24) .. (645.33,123.53) -- (645.33,163.13) .. controls (645.33,170.42) and (639.42,176.33) .. (632.13,176.33) -- (388.53,176.33) .. controls (381.24,176.33) and (375.33,170.42) .. (375.33,163.13) -- cycle ;
\draw    (150.7,105.53) .. controls (170.64,70.55) and (273.33,93.58) .. (312.33,64.33) ;
\draw [shift={(149.33,108.33)}, rotate = 292.31] [fill={rgb, 255:red, 0; green, 0; blue, 0 }  ][line width=0.08]  [draw opacity=0] (8.93,-4.29) -- (0,0) -- (8.93,4.29) -- cycle    ;
\draw    (510.95,105.44) .. controls (490.76,70.46) and (386.81,93.49) .. (347.33,64.24) ;
\draw [shift={(512.33,108.24)}, rotate = 247.45] [fill={rgb, 255:red, 0; green, 0; blue, 0 }  ][line width=0.08]  [draw opacity=0] (8.93,-4.29) -- (0,0) -- (8.93,4.29) -- cycle    ;
\draw   (139.33,189.33) -- (147.83,179.33) -- (156.33,189.33) -- (152.08,189.33) -- (152.08,209.33) -- (156.33,209.33) -- (147.83,219.33) -- (139.33,209.33) -- (143.58,209.33) -- (143.58,189.33) -- cycle ;
\draw   (16.33,238.87) .. controls (16.33,231.58) and (22.24,225.67) .. (29.53,225.67) -- (273.13,225.67) .. controls (280.42,225.67) and (286.33,231.58) .. (286.33,238.87) -- (286.33,278.47) .. controls (286.33,285.76) and (280.42,291.67) .. (273.13,291.67) -- (29.53,291.67) .. controls (22.24,291.67) and (16.33,285.76) .. (16.33,278.47) -- cycle ;
\draw   (376.33,239.53) .. controls (376.33,232.24) and (382.24,226.33) .. (389.53,226.33) -- (633.13,226.33) .. controls (640.42,226.33) and (646.33,232.24) .. (646.33,239.53) -- (646.33,279.13) .. controls (646.33,286.42) and (640.42,292.33) .. (633.13,292.33) -- (389.53,292.33) .. controls (382.24,292.33) and (376.33,286.42) .. (376.33,279.13) -- cycle ;
\draw   (505.33,190) -- (513.83,180) -- (522.33,190) -- (518.08,190) -- (518.08,210) -- (522.33,210) -- (513.83,220) -- (505.33,210) -- (509.58,210) -- (509.58,190) -- cycle ;
\draw   (110.33,362) .. controls (110.33,354.82) and (116.15,349) .. (123.33,349) -- (537.33,349) .. controls (544.51,349) and (550.33,354.82) .. (550.33,362) -- (550.33,401) .. controls (550.33,408.18) and (544.51,414) .. (537.33,414) -- (123.33,414) .. controls (116.15,414) and (110.33,408.18) .. (110.33,401) -- cycle ;
\draw    (149.33,295.33) .. controls (165.09,335.49) and (269.14,310.74) .. (310.49,339.96) ;
\draw [shift={(312.33,341.33)}, rotate = 218.1] [fill={rgb, 255:red, 0; green, 0; blue, 0 }  ][line width=0.08]  [draw opacity=0] (8.93,-4.29) -- (0,0) -- (8.93,4.29) -- cycle    ;
\draw    (515.33,295.33) .. controls (499.28,335.49) and (393.32,310.74) .. (351.21,339.96) ;
\draw [shift={(349.33,341.33)}, rotate = 322.41] [fill={rgb, 255:red, 0; green, 0; blue, 0 }  ][line width=0.08]  [draw opacity=0] (8.93,-4.29) -- (0,0) -- (8.93,4.29) -- cycle    ;

\draw (329.83,42.5) node   [align=left] {\textbf{{\large Modern Machine Learning}}};
\draw (151.27,142.46) node   [align=left] {\begin{minipage}[lt]{174.66pt}\setlength\topsep{0pt}
\begin{center}
{\small Hessian matrix of the finite-sum objectives}\\{\small often exhibits a low-rank structure when }\\{\small restricted to a batch of data}
\end{center}

\end{minipage}};
\draw (509.57,143.5) node   [align=left] {\begin{minipage}[lt]{172.14pt}\setlength\topsep{0pt}
\begin{center}
{\small Stochastic zeroth-order optimization has}\\{\small become increasingly important, especially }\\{\small in fine-tuning large language models (LLMs)}
\end{center}

\end{minipage}};
\draw (146.7,259.67) node   [align=left] {\begin{minipage}[lt]{122.59pt}\setlength\topsep{0pt}
\begin{center}
\textbf{{\fontsize{0.87em}{1.04em}\selectfont Intrinsic Low-rank Hessians }}\\\textbf{{\fontsize{0.87em}{1.04em}\selectfont on a Batch of Data}}
\end{center}

\end{minipage}};
\draw (511.33,259.33) node   [align=left] {\begin{minipage}[lt]{163.82pt}\setlength\topsep{0pt}
\begin{center}
\textbf{{\fontsize{0.87em}{1.04em}\selectfont  Stochastic Zeroth-order Optimization }}\\\textbf{{\fontsize{0.87em}{1.04em}\selectfont in Modern Machine Learning}}
\end{center}

\end{minipage}};
\draw (329.66,380.2) node   [align=left] {\textbf{Question:} can we exploit low-rank Hessian structures to improve\\the efficiency of stochastic zeroth-order optimization algorithms?};

\end{tikzpicture}

%% file: tex/algo-brief.tex
\section{The Algorithm Procedure and Main Results} 
\label{sec:brief}

Many widely used zeroth-order algorithms borrow from derivative-based methods, by estimating the derivatives first. 
As is well-known, second-order information can better capture the curvature of the objective, and thus accelerate the convergence rate. To this end, we study the zeroth-order stochastic cubic Newton method, whose algorithmic framework is summarized  in Algorithm \ref{alg:high-level}.


\begin{algorithm}[h] 
    \caption{High-level description of zeroth-order cubic Newton method} 
    \label{alg:high-level} 
    \begin{algorithmic}[1]  
        \State \textbf{Goal.} A finite-sum-type objective $ \int_{\xi \in \Theta} f (\x;\xi) \mu (d\xi) $. 
        \State \textbf{Input parameters.} $T$, $m_1$, $m_2$. /* $T$ is the total time horizon; $m_1$ (resp. $m_2$) describes number of function evaluations used for gradient (resp. Hessian) estimation at each step.~\footnote{$m_1$ and $m_2$ does not equal the number of function evaluations needed exactly. We will further specify the exact role of $m_1$ and $m_2$ as we expand the details of the algorithm.} */
        \Statex /* There are additional input parameters, as we will detail in subsequent sections. */
        \State \textbf{Initialize.} $\x_0 \in \R^n$. 
        \FOR{$t = 0,1,2,\cdots, T$} 
            \State Generate $iid$ $ \{ \xi_t^{i,1} \}_{i\in [m_1]} $ following $\mu$. Obtain gradient estimator $\g_t$ for $ {\nabla} F (\x_t) $, using function evaluations for $ \{ f (\x_t; \xi_t^{i,1}) \}_{i\in [m_1]}$.
            \State Generate $iid$ $ \{ \xi_t^{i,2} \}_{i\in [m_2]} $ following $\mu$. Obtain Hessian estimator $\wh{H}_t$ for $ {\nabla}^2 F (\x_t) $, using function evaluations for $ \{ f (\x_t; \xi_t^{i,2}) \}_{i\in [m_2]}$.
            \State Perform a cubic Newton step with estimated gradient and Hessian, and obtain $\x_{t+1}$. 
        \ENDFOR  
        

    \end{algorithmic} 
\end{algorithm} 




\subsection{Low-rank Hessian Estimation}
\label{sec:low-rank-brief}



We initiate our approach by formulating a finite-difference discretization as a matrix-represented measurement operator. 
In stochastic zeroth-order optimization, a $\xi$ is randomly selected at each step, and function evaluations of $f (\cdot; \xi)$ is performed. Once a $\xi$ is selected, the function $f (\cdot; \xi)$ is fixed. Throughout this subsection, we will omit the explicit mention of $\xi$ and refer to $f (\cdot; \xi)$ as simply $f (\cdot)$ when there is no confusion.

Below, we recover a low-rank Hessian estimate from two distinct undersampled measurements using nuclear norm minimization.
\subsubsection{Recovery via spherical measurement}
The Hessian of a function $f : \R^n \to \R$ at a given point $\x$ can be estimated as follows \citep{wang2022hess,10.1093/imaiai/iaad014} 
\begin{align}
    \wh{\nabla}_{\u, \v, \delta}^2 f(\x) 
    \!:= \!
    n^2 \frac{ f (\x \!+ \!\delta \v \!+ \!\delta \u ) \!-\! f (\x \!-\! \delta \v \!+\! \delta \u ) \!-\! f (\x \!+\! \delta \v \!-\! \delta \u ) \!+\! f (\x \!-\! \delta \v \!-\! \delta \u ) }{ 4 \delta^2 } \u \v^\top , \label{eq:hess-est} 
\end{align} 
where $\delta$ denotes the finite-difference granularity, and $ \u, \v \in \R^n$ represent finite-difference directions. The choice of probability distributions for $\u$ and $\v$ governs the resulting Hessian estimator. For example, $ \u, \v $ can be independent random vectors uniformly sampled from the canonical basis $\{ \e_1, \e_2, \cdots, \e_n \}$, or uniformly distributed over the unit sphere in $\R^n$.

We first show that the finite difference Hessian estimator (\ref{eq:hess-est}) is approximately a matrix measurement, in Proposition \ref{prop:measure} and Lemma \ref{lem:measurement}. 

\begin{proposition} 
    \label{prop:measure} 
    Consider an estimator defined in (\ref{eq:hess-est}). Let the underlying function $f$ be three-time continuously differentiable. Let $\u,\v \in \R^n$ be two random vectors. Then for any fixed $\x \in \R^n$, 
    \begin{align*}  
        \wh{\nabla}_{\u, \v, \delta}^2 f(\x) 
        \to n^2 \u \u^\top \nabla^2 f (\x) \v \v^\top , \; \; a.s.,\;\; \text{as} \; \delta \to 0_+.
    \end{align*}  
\end{proposition} 

\begin{proof}[Proof of Proposition \ref{prop:measure}] 
    Let $ \( \Omega , \mathcal{F} , \mu \) $ be the probability space over which $ \u, \v $ are defined. 
    Let $\omega \in \Omega$ be any element in the sample space of the random vectors $\u,\v$. 
    By Taylor's Theorem and that the Hessian matrix is symmetric, we have 
    \begin{align*} 
        & \; \wh{\nabla}_{\u (\omega), \v(\omega), \delta}^2 f(\x) \\ 
        =& \; 
        \frac{n^2}{4} 
        \( \v (\omega) + \u (\omega) \)^\top \nabla^2 f (\x) \( \v (\omega) + \u (\omega) \) \u (\omega) \v (\omega)^\top \\
        &- \frac{n^2}{4}
        \( \v (\omega) - \u (\omega) \)^\top \nabla^2 f (\x) \( \v (\omega) - \u (\omega) \)  \u (\omega) \v (\omega)^\top \\
        &+ \mathcal{O} \( \delta \( \| \v (\omega) \| + \| \u (\omega) \| \)^3 \) \\ 
        =& \; 
        n^2 \u(\omega)^\top \nabla^2 f (\x) \v (\omega) \u (\omega) \v (\omega)^\top + \mathcal{O} \( \delta \( \| \v (\omega) \| + \| \u (\omega) \| \)^3 \) 
        \\
        =&\; 
        n^2 \u (\omega) \u (\omega)^\top \nabla^2 f (\x) \v(\omega) \v(\omega)^\top + \mathcal{O} \( \delta \( \| \v (\omega) \| + \| \u (\omega) \| \)^3 \) , 
    \end{align*} 
    where the last equality uses that $ \u(\omega)^\top \nabla^2 f (\x) \v (\omega) $ is a real number and can thus be moved around as a whole. 
    
    As $\delta \to 0_+$, the estimator (\ref{eq:hess-est}) converges to $ n^2 \u (\omega) \u (\omega)^\top \nabla^2 f (\x) \v (\omega) \v^\top (\omega) $ for any sample point $\omega$. This concludes the proof. \if\proofmark1{\hfill $\square$}\fi
\end{proof}

When the finite-difference granularity is non-vanishing, we have the following lemma that quantifies the measurement error. 

\begin{lemma}   
    \label{lem:measurement} 
    Let $ f $ have $L_2$-Lipschitz Hessian: $\|\nabla^2 f (\x) - \nabla^2 f (\y) \| \le L_2 \| \x- \y \|$, for all $\x, \y \in \R^n$. 
    Consider the estimator defined in (\ref{eq:hess-est}). Let $M < \frac{n(n+1)}{2}$.~\footnote{Since the Hessian is symmetric, the Hessian estimation problem is trivial when number of measurements $ M \ge \frac{n(n+1)}{2}$.} 
    For each $ \x \in \R^n $, there exists a real symmetric random matrix $H$ of size $n \times n$ such that (i) $H$ is determined by $ \{ \u_i , \v_i \}_{i=1}^M $, and (ii) $\| H - \nabla^2 f (\x) \| {\le} 4 L_2 \delta $ almost surely, and (iii)
    \begin{align*}
        \wh{\nabla}_{\u_i , \v_i , \delta}^2 f (\x)  = {n^2}\u_i \u_i^\top H \v_i \v_i^\top , \quad  i = 1,2,\cdots, M, 
    \end{align*} 
    where $\u_i, \v_i \overset{iid}{\sim} \text{Unif} (\mathbb{S}^{n-1})$. 
\end{lemma}

\begin{proof}[Proof of Lemma \ref{lem:measurement}]
    By Taylor's theorem, we know that for any $\u ,\v \in \mathbb{S}^{n-1}$, there exists $ \z^1 $ depending on $\u, \v, \delta$ such that 
    \begin{align*} 
        f (\x + \delta \u + \delta \v) = f (\x) + \delta \< \nabla f (\x) , \u + \v \> + \frac{\delta^2}{2} \< \u + \v, \nabla^2 f (\z^1) \(\u+\v\) \> , 
    \end{align*} 
    and $ \| \z^1 - \x \| \le 2\delta $. 
    Similarly, we have 
    \begin{align*} 
        f (\x + \delta \u - \delta \v) =&\; f (\x) + \delta \< \nabla f (\x) , \u - \v \> + \frac{\delta^2}{2} \< \u - \v, \nabla^2 f (\z^2) \(\u-\v\) \> , \\ 
        f (\x - \delta \u + \delta \v) =&\; f (\x) + \delta \< \nabla f (\x) , -\u + \v \> + \frac{\delta^2}{2} \< \u - \v, \nabla^2 f (\z^3) \(\u-\v\) \> , \\ 
        f (\x - \delta \u - \delta \v) =&\; f (\x) + \delta \< \nabla f (\x) , -\u - \v \> + \frac{\delta^2}{2} \< \u + \v, \nabla^2 f (\z^4) \(\u+\v\) \> , 
    \end{align*} 
    where $ \| \z^i - \x \| \le 2\delta $ ($i=2,3,4$). When $\u, \v$ are uniformly randomly distributed over $ \mathbb{S}^{n-1} $, then $ \nabla^2 f (\z^i) $ ($i=1,2,3,4$) is a random variable contained in the $\sigma$-field of $\u $ and $\v$. 

    We define the random symmetric matrix $H$ as follows.
    Let $ (\Omega, \mathcal{F}, \Pr) $ be the probability space over which $ \{ \u_i, \v_i \}_{i=1}^M $ are defined. 
    For any $\omega \in \Omega$, 
    we define $ H(\omega) $ such that 
    \begin{align} 
        8 \u_i(\omega)^\top H (\omega) \v_i (\omega) 
        =& \;  
        \( \u_i(\omega) + \v_i (\omega) \)^\top \nabla^2 f ( \z_i^1 (\omega) ) \( \u_i(\omega) + \v_i (\omega) \) \nonumber \\ 
        &-
        \( \u_i(\omega) - \v_i (\omega) \)^\top \nabla^2 f ( \z_i^2 (\omega) ) \( \u_i(\omega) - \v_i (\omega) \) \nonumber \\ 
        &-
        \( -\u_i(\omega) + \v_i (\omega) \)^\top \nabla^2 f ( \z_i^3 (\omega) ) \( -\u_i(\omega) + \v_i (\omega) \) \nonumber \\ 
        &+
        \( -\u_i(\omega) - \v_i (\omega) \)^\top \nabla^2 f ( \z_i^4 (\omega) ) \( - \u_i(\omega) - \v_i (\omega) \) . \label{eq:equ}
    \end{align} 

    There are $ M $ equations (linear in the entries of $H (\omega)$) in the form of (\ref{eq:equ}) and the degree of freedom of $H(\omega)$ is $ \frac{n(n+1)}{2} $. Thus when $M < \frac{n(n+1)}{2}$, we can find a symmetric matrix $H(\omega)$ for each point $\omega$ in the sample space. When $M < \frac{n(n+1)}{2}$, this matrix $H(\omega)$ is not unique. In this case, we use an arbitrary (deterministic) rule to pick an $H (\omega)$ that satisfies all equations in (\ref{eq:equ}). Note that all randomness comes from the random measurement vectors $\u_i$ and $\v_i$. 
    
    \if\proofmark1{\hfill $\square$}\fi
\end{proof} 

With Proposition \ref{prop:measure} and Lemma \ref{lem:measurement} in place, we can formulate Hessian estimation task as a matrix recovery problem. Specifically, the Hessian estimation seeks to solve the following convex program 
\begin{align} 
    \min_{\wh{H} \in \R^{n \times n} } \| \wh{H} \|_1 , 
    \text{ subject to } 
    \S \wh{H} = \frac{1}{M} \sum_{i=1}^M 
    \wh{\nabla}_{\u_i, \v_i, \delta}^2 f (\x)
    = 
    \frac{1}{M} \sum_{i=1}^M n^2 \u_i \u_i^\top H \v_i \v_i^\top, 
    \label{eq:goal} 
\end{align} 
where $\|\cdot\|_1$ denotes trace norm, $ \u_i , \v_i $ are $iid$ random measurement vectors, 
$\S := \frac{1}{M} \sum_{i=1}^M \P_i$ with 
\begin{align}
    \P_i A = n^2\u_i\u_i^\top A \v_i\v_i^\top \label{eq:def-measurement}
\end{align}
for any $A \in \R^{n \times n}$, and $H$ is a random symmetric matrix determined by $ \{ \u_i, \v_i \}_{i=1}^M $ that is implicitly obtained via Eq. \ref{eq:hess-est} (See Lemma \ref{lem:measurement}).

\subsubsection{Recovery via Gaussian measurement}
Different from the finite difference Hessian estimator (\ref{eq:hess-est}), we can obtain an approximately quadratic form with respect to the real Hessian \citep{balasubramanian2022zeroth,li2022stochastic}
\begin{align}
    \label{eq:quadric-form}
    Q_{\delta,\mathbf{a}}(\x):=
     \frac{f(\x+\delta \mathbf{a})+f(\x-\delta \mathbf{a})-2f(\x)}{\delta^2}  
     \approx 
     \mathbf{a}^\top \nabla^2 f (\x) \mathbf{a} ,
\end{align}
where $\delta$ denotes the  finite-difference step size, and $\mathbf{a}\in \mathbb{R}^n$ represent finite-difference direction vector.

In this subsection, we based our analysis on the approximately quadric form (\ref{eq:quadric-form}).  Let $\mathbf{a} \sim \mathcal{N}(\mathbf{0},I_n)$. 
Similar to Proposition \ref{prop:measure}, by Taylor expansion, as $\delta \to 0^+$, it holds almost surely that 
\begin{align*}
    \frac{f(\x+\delta \mathbf{a})+f(\x-\delta \mathbf{a})-2f(\x)}{\delta^2} \overset{a.s.}{\to}\ \mathbf{a}^\top \nabla ^2 f(\x) \mathbf{a}. 
\end{align*}
Besides, there exist $\y_1$ and $\y_2$ (depending on $\x,\mathbf{a}$, and $\delta$) such that  
\begin{align*}
    \frac{f(\x+\delta \mathbf{a})+f(\x-\delta \mathbf{a})-2f(\x)}{\delta^2} =\frac{1}{2} \mathbf{a}^\top [\nabla^2 f(\y_1)+\nabla^2 f(\y_2)] \mathbf{a},
\end{align*}
where $\|\y_i-\x\| \leq \delta \|\mathbf{a}\|$ for $i=1,2$.  
Formally, we can express
\begin{align*}
    \frac{1}{2} \mathbf{a}^\top [\nabla^2 f(\y_1)+\nabla^2 f(\y_2)] \mathbf{a} \quad \text{as} \quad \mathbf{a}^\top H^\prime \mathbf{a},
\end{align*}
where $H^\prime$ is a symmetric random matrix depending on $a,\delta$, and $\x$. Note that we use $H^\prime$ to distinct the symmetric random matrix $H$ obtained by spherical measurement. 

Similarly, we may formulate the Hessian estimation task as a matrix recovery problem. Specifically, we  solve the following convex program:
\begin{align}
    \label{eq:goal_gaussian}
    \min_{\wh{H} \in \R^{n \times n} } \!\| \wh{H} \|_1 , 
    \text{ subject to } 
     \S^\prime \wh{H}
    \!=\! 
    \frac{1}{M} \!\sum_{i=1}^M \!\frac{f(\x\!+\!\delta \mathbf{a}_i)\!+\!f(\x\!-\!\delta \mathbf{a}_i)\!-\!2f(\x)}{\delta^2}
    \!=\!
    \frac{1}{M} \!\sum_{i=1}^M \!\mathbf{a}_i^\top \!H^\prime \mathbf{a}_i, 
\end{align}
where $\mathbf{a}_1,\mathbf{a}_2,\cdots,\mathbf{a}_M$ are $i.i.d.$ standard Gaussian vectors, and $\S^\prime:=\frac{1}{M} \sum_{i=1}^M \mathcal{Q}_i $ with 
\begin{align*}
    \mathcal{Q}_i A =\mathbf{a}_i ^\top A \mathbf{a}_i,\quad \forall\; A \in \mathbb{R}^n. 
\end{align*}

\subsubsection{Theoretical guarantee of the recovery}

Both methods recover high-accuracy estimators with high probability from undersampled measurements. Before presenting the formal theorem, the following remark is essential.

\begin{remark}
    \label{remark:omit_finite_difference_error}
    As established in Lemma \ref{lem:measurement}, spherical measurements enable exact determination of the finite-difference gap between the symmetric random matrix $H$ and the true Hessian $\nabla^2 f(\x)$. For Gaussian measurements, this gap is similarly controllable. Moreover, the finite-difference gap in gradient estimation admits straightforward control. Therefore, for brevity and analytical clarity, we omit all finite-difference gaps in subsequent analysis. Their inclusion would introduce redundancy without affecting the fundamental outcomes for sampling complexity or convergence rates.
\end{remark}

Building on \cite{kueng2017low,wang2024zeroth}, we now present the main recovery theorem.

\begin{theorem}[Spherical measurement] 
    \label{thm:main} 
    Consider the problem (\ref{eq:goal}) with a fixed $\x \in \R^n$. Let $f$ have $L_2$-Lipschitz Hessian: There exists $L_2 > 0$, such that $ \| \nabla^2 f (\x) - \nabla^2 f (\y)  \| \le L_2 \| \x - \y \| $ for any $\x,\y \in \R^n$. Let the sampler $ \S = \frac{1}{M} \sum_{i=1}^M \P_i $ be constructed with $ \P_i : A \mapsto n^2 \u_i\u_i^\top A \v_i \v_i^\top $ and $ \u_i, \v_i \overset{iid}{\sim} \text{Unif} ( \mathbb{S}^{n-1}) $. 
    If the Hessian matrix $\nabla^2 f (\x)$ is rank-$r$. 
    Then there exists an absolute constant $K$, such that for any $\beta \in (0,1)$, if the number of samples $M \ge K\cdot n r^2 \log (n) \log (n/\beta) $, then 
    with probability larger than $1 - \beta $, the solution to (\ref{eq:goal}), denoted by $\wh{H}$, satisfies $\wh{H}=\nabla^2 f(\x) $. 
\end{theorem} 

\begin{theorem}[Gaussian measurement] 
    \label{thm:main_gaussian} 
    Consider the problem (\ref{eq:goal_gaussian}) with a fixed $\x \in \R^n$. Let $f$ have $L_2$-Lipschitz Hessian: There exists $L_2 > 0$, such that $ \| \nabla^2 f (\x) - \nabla^2 f (\y)  \| \le L_2 \| \x - \y \| $ for any $\x,\y \in \R^n$. Let the sampler $ \S^\prime = \frac{1}{M} \sum_{i=1}^M \mathcal{Q}_i $ be constructed with $ \mathcal{Q}_i : A \mapsto  \mathbf{a}_i^\top A \mathbf{a}_i $, and $ \mathbf{a}_i \overset{iid}{\sim} \mathcal{N}(\mathbf{0},I_n)$. 
    If the Hessian matrix $\nabla^2 f (\x)$ is rank-$r$. 
    Then there exist  absolute constant $K_1$ and $K_2$, such that  if the number of samples $M \ge K_1nr $, then 
    with probability larger than $1 - e^{-K_2 M} $, the solution to (\ref{eq:goal_gaussian}), denoted by $\wh{H}$, satisfies $\wh{H}=\nabla^2 f(\x) $. 
\end{theorem}

\begin{remark}
    When the Hessian matrix is approximately low-rank rather than strictly low-rank---that is, when there exists an  $H_r$ with $rank (H_r) = r$ such that $ \| \nabla^2 f (\x) - H_r \|_1 \le \epsilon $---Theorem \ref{thm:main} and \ref{thm:main_gaussian}
    yield
     $\|\wh{H}-\nabla^2 f(\x)\| \leq \epsilon$.
\end{remark}

At this point, we have established a method for estimating $\nabla^2 f (\x; \xi)$ for given values of $\x$ and $\xi$. In the context of stochastic optimization, we sample a set of $\xi_i$'s and estimate the gradient and Hessian using the functions $f (\cdot; \xi_i)$. This estimation procedure is outlined in Algorithm \ref{alg:est}. Additionally, Algorithm \ref{alg:est} includes a well-understood method for gradient estimation in (\ref{eq:grad-est}). 


\begin{algorithm}[h] 
    \caption{Gradient and Hessian Estimation at Step $t$} 
    \label{alg:est} 
    \begin{algorithmic}[1]  
        \State \textbf{Input parameters.} $m_1$, $m_2$, $M$, {finite-difference granularity $\delta$.} /* $m_1, m_2, M$ and the dimensionality $n$ determines total number of functions evaluations used to estimate the gradient and the Hessian. */ 
        \State Let $\x_t$ be the current iterate point. 
        \State Generate $iid$ $ \{ \xi_t^{i,1} \}_{i\in [m_1]} $. 
        \State Estimate the gradient $ {\nabla} F (\x_t)$ by a vector $ \g_t $, such that for each $j \in [n], $ 
        \begin{align} 
            \< \g_t , \e_j \> = \frac{1}{m_1} \sum_{i=1}^{m_1} \frac{ f ( \x_t+\e_j\delta;\xi_t^{i,1})-f(\x_t-\e_j\delta;\xi_t^{i,1})}{2\delta} . 
            \label{eq:grad-est} 
        \end{align} 

        \Statex {Generate Hessian estimator by spherical measurement (\textbf{Option 1}) or Gaussian measurement (\textbf{Option 2}).}
        \Statex \textbf{Option 1:}
        \State Generate $iid$ $ \{ \xi_t^{i,2} \}_{i\in [m_2]} $, and $\u_{i,j}, \v_{i,j} \overset{iid}{\sim} \mathrm{Unif} (\mathbb{S}^{d-1})$ for $i \in [m_2]$ and $ j \in [M] $. 
        \State For each $i$, obtain spherical Hessian measurements $ \{ \wh{\nabla}_{\u_{i,j}, \v_{i,j}, \delta}^2 f(\x_t; \xi_t^{i,2})  \}_{j \in [ M ]} $ {by (\ref{eq:hess-est})}, and obtain a Hessian estimation $ \wh{H}_t( \xi_t^{i,2} ) $ via solving the convex program (\ref{eq:goal}). Define $ \wh{H}_t := \frac{1}{m_2} \sum_{i=1}^{m_2} \wh{H}_t( \xi_t^{i,2} )  $. 
        \Statex \textbf{Option 2:}
        \State Generate $iid$ $ \{ \xi_t^{i,2} \}_{i\in [m_2]} $, and $\mathbf{a}_{i,j} \overset{iid}{\sim} \mathcal{N}(\mathbf{0},I_n)$ for $i \in [m_2]$ and $ j \in [M] $. 
        \State For each $i$, obtain Gaussian Hessian measurements $ \{ Q_{\delta,\mathbf{a}}(\x;\xi_t^{i,2})  \}_{j \in [ M ]} $ {by (\ref{eq:quadric-form})}, and obtain a Hessian estimation  $ \wh{H}_t^\prime( \xi_t^{i,2} ) $ via solving the convex program (\ref{eq:goal_gaussian}). Define $ \wh{H}_t^\prime := \frac{1}{m_2} \sum_{i=1}^{m_2} \wh{H}_t^\prime( \xi_t^{i,2} ) $.

    \end{algorithmic} 
\end{algorithm} 

In the following analysis, we focus exclusively on spherical measurements since the theoretical guarantees for both recovery methods are similar.
The estimation guarantee of $ \wh{H}_t( \xi_t^{i,2} )  $ in Algorithm \ref{alg:est} over $ t = 0,1,2,\cdots,T-1$ is given below in Theorem \ref{thm:hess-est}. 


\begin{theorem}
    \label{thm:hess-est} 
    Let Assumptions \ref{as:H lip cts } and \ref{assumption:low-rank} be true. Let $ M \!\ge\! K nr^2 \log n \log (nT/\beta) $. 
    For any $\beta \in (0,1)$ and $\delta > 0$, with probability exceeding $1 - \beta$, the output of Algorithm \ref{alg:est} satisfies $  \wh{H}_t( \xi_t^{i,2} )  = \nabla^2 f (\x_t;\xi_t^{i,2} ) $ for  $t = 0,1,2,\cdots,T-1$. 
\end{theorem} 

\begin{proof}[Proof of Theorem \ref{thm:hess-est}]
    Using Theorem \ref{thm:main} and taking a union bound over $ t = 0,1,2,\cdots, T-1 $ conclude the proof. \if\proofmark1{\hfill $\square$}\fi
\end{proof}

With all the above results in place, we move on to the performance guarantee of the stochastic zeroth-order cubic Newton method for the low-rank scenario. 

\subsection{Zeroth-order Stochastic Cubic Newton Method}

In Section \ref{sec:low-rank-brief}, we have set up the gradient and Hessian estimators at $\x_t$. With this in place, 
we can state a more detailed version of Algorithm \ref{alg:high-level}, which can be found in Algorithm \ref{alg:cubic_newton_method}. The performance guarantee for Algorithm \ref{alg:cubic_newton_method} is found below in Theorem \ref{thm:epsilon local optima-0}. Before presenting Theorem \ref{thm:epsilon local optima-0}, we state the following standard assumptions for stochastic optimization problems.  

\begin{algorithm}[h!] 
    \caption{Zeroth-order Stochastic Cubic Newton Method} 
    \label{alg:cubic_newton_method} 
    \begin{algorithmic}[1]  
        \State \textbf{Input parameters:} $m_1$, $m_2$, $M$, $\alpha$, number of iterations $T$, {$\delta$.} /* $m_1$, $m_2$, $M$, $\delta$ as defined in Algorithm \ref{alg:est}. $\alpha$ is a regularization parameter. */
        \State \textbf{Initialize:} $\x_0 \in \mathbb{R}^n$. 
        \FOR{$t=0,1,\cdots,T-1$}
        \State Generate the corresponding gradient estimator $\g_t$ and Hessian estimator $\wh{H}_t$ of function $F$ at point $\x_t$ according to Algorithm \ref{alg:est}. /* The gradient estimator $ \g_t $ depends on $m_1$ and $\delta$; the Hessian estimator $\wh{H}_t$ depends on $ m_2, M $ and $\delta$. */
        \State Compute 
        \begin{align*}
            \x_{t+1} = \arg\min_{\y} \wh{f}_{t} ( \y ),
        \end{align*}
        where
        \begin{align} 
      \wh{f}_{t} ( \y ) 
    := 
    \< \g_t , \y - \x_t \> + \frac{1}{2} \< \y - \x_t, \wh{H}_t (\y - \x_t) \> + \frac{\alpha}{6} \| \y - \x_t \|^3, \label{eq:def-f-hat} 
       \end{align} 
       and $\g_t$ and $ \wh{H}_t $ are obtained from Algorithm \ref{alg:est}. 
       \ENDFOR
       \State \textbf{Output:} $\x_{R+1}$, where $R$ is a random variable uniformly distributed  on $\{0,1, 2,\cdots,T-1\}$. 
    
    \end{algorithmic} 
\end{algorithm} 



\begin{assumption}
    \label{assumption:grad-var}
    There exists $\sigma_1 \ge0$ such that
    \begin{align*}
        \E \| \nabla f (\x; \xi) - \nabla F (\x) \|^2 \le \sigma_1^2 , \quad \forall \x \in \R^n, 
    \end{align*}
    where the expectation is taken with respect to $\xi \sim \mu$, and $\mu$ is the measure for the objective (\ref{eq:goal}). 
\end{assumption}

\begin{assumption}
    \label{assumption:hess-var}
    There exist constants $\sigma_2,\tau_2 \ge 0$, such that 
    \begin{align*}
        \E \| \nabla^2 f (\x; \xi) - \nabla^2 F (\x) \|^2 \le \sigma_2^2 , \quad 
        \E \| \nabla^2 f (\x; \xi) - \nabla^2 F (\x) \|^4 \le \tau_2^4 , \quad \forall \x \in \R^n, 
    \end{align*} 
    where the expectation is taken with respect to $\xi \sim \mu$. 
\end{assumption}

\begin{remark} 
    In previous work \citep{balasubramanian2021zeroth}, an $L_1$-smoothness condition is assumed. That is, there exists $L_1 > 0$, such that $ \| \nabla f (\x; \xi) - \nabla f (\y; \xi) \| \le L_1 \| \x - \y \|,$ for all $\x, \y \in \R^n$ and all $\xi \in \Theta$. With this condition, the Hessian is bounded: $ \| \nabla^2 f (\x; \xi) \| \le L_1 $ for all $\x \in \R^n$ and $\xi \in \Theta$. This absolute boundedness condition, previously assumed in \citep{balasubramanian2021zeroth}, is stronger than our Assumption \ref{assumption:hess-var}. 
\end{remark} 

Now we are ready to present the convergence result.

\begin{theorem}
    \label{thm:epsilon local optima-0} 
    Consider a stochastic optimization problem (\ref{eq:goal}) defined over $\R^n$. Let Assumption \ref{as:H lip cts } with Lipschitz constant $L_2$ and  Assumption \ref{assumption:low-rank} be true, and let $ r $ be an upper bound for the rank of the Hessian $\nabla^2 f (\x; \xi)$. Suppose $F^*: = \min_\x F (\x) > -\infty$. 
    For any $\eta \in (0,1)$ and $\beta \in (0,1)$, let 
    \begin{align} 
        \begin{split} 
            \alpha=L_2, \quad T
            \asymp \frac{\sqrt{L_2}\( F ( \x_0)-F^*\) }{ \eta^\frac{3}{2}}, \quad 
            m_1 \asymp \frac{1}{\eta^2}  \\
            m_2 \asymp \frac{n}{ \eta L_2 },  \quad \text{and}  \quad M \asymp n r^2 \log n \log \(\frac{2nT}{\beta \eta L_2 }\) . 
        \end{split} 
        \label{eq:para-0}
    \end{align}

    With Assumptions \ref{assumption:grad-var} 
    and \ref{assumption:hess-var}, the sequence $\{ \x_t \}_t$ governed by Algorithm \ref{alg:cubic_newton_method}, with input parameters specified as in (\ref{eq:para-0}), satisfies  
    \begin{align} 
        \sqrt{\eta} \geq \max\left\{\sqrt{ \E\[  \|\nabla F(\x_{R+1})\| | \mathcal{E}\]},\,-\frac{2}{3\sqrt{L_2}}  \E\[\lambda_{\min}\(\nabla^2 F\(\x_{R+1}\)\) | \mathcal{E} \] \right\}, \label{eq:2-order}
    \end{align} 
    where $\mathcal{E}$ is an event that holds with probability exceeding $1 - \beta$, and $R$ is a random variable  uniformly distributed on $\{0,1,\cdots,T-1\}$ independent of all other randomness. 
\end{theorem} 

\begin{remark} 
    By Theorem \ref{thm:epsilon local optima-0}, to obtain a second order $\eta$-stationary point for the nonconvex problem, the number of samples required to estimate gradient and Hessian are $Tm_1 \cdot \mathcal{O}(n)=\mathcal{O}\(\frac{n}{\eta^{\frac{7}{2}}}\)$ and ${TMm_2}=\widetilde{\mathcal{O}}\(\frac{n^2 r^2}{\eta^{\frac{5}{2}}}\)$, respectively. 
\end{remark}

In the literature \citep{NEURIPS2018_db191505,balasubramanian2021zeroth}, a point satisfying conditions in (\ref{eq:2-order}) is referred to as a second-order stationary point, whose formal definition is given below.


\begin{definition} 
    \label{def:second-order} 
    Let $\eta > 0$ be an arbitrary number. 
    A point $\x$ is called a second-order $\eta$-stationary point for function $F$ if 
    \begin{align*} 
        \sqrt{\eta} \geq \Omega \( \max\left\{ \sqrt{ \|\nabla F(\x )\| } ,\,-\lambda_{\min}\(\nabla^2 F\( \x \) \)  \right\} \) . 
    \end{align*} 
    A random point $\x$ is called a second-order $\eta$-stationary point in the expectation sense if 
    \begin{align*} 
        \sqrt{\eta} \geq \Omega \( \max\left\{ \sqrt{ \E \|\nabla F (\x )\| } ,\,- \E \[ \lambda_{\min}\(\nabla^2 F \( \x \) \) \] \right\} \) . 
    \end{align*} 
\end{definition}

%% file: tex/algo.tex
\section{Algorithm Analysis}  
\label{sec:algo-analysis}


We turn our focus to sample complexity analysis of Algorithm \ref{alg:cubic_newton_method}. 

\subsection{Gradient and Hessian Estimation}
\label{subsec:grad and hess esti} 

We begin the proof of Theorem \ref{thm:epsilon local optima-0} by deriving moment bounds for the gradient and Hessian estimators, below in  Theorem \ref{thm:moment-bound}.

\begin{theorem} 
    \label{thm:moment-bound} 
    Instate Assumptions \ref{as:H lip cts }, \ref{assumption:low-rank}, \ref{assumption:grad-var} and \ref{assumption:hess-var}. Let $\g_t$ and $\wh{H}_t$ be the estimators defined as in Algorithm \ref{alg:est}. 
    Then we have 
    \begin{align*}
        \E \[ \| \nabla F (\x_t) - \g_t \|^{3/2} \]  
        \le 
        \(\frac{\sigma_1^2}{m_1}\)^{3/4} . 
    \end{align*}
    For any $\beta \in (0,1)$, let $M \ge K \cdot n r^2 \log (n) \log (T m_2 n /\beta),$ 
    then there exists an event $\mathcal{E}$ such that $ \Pr \( \mathcal{E} \) \ge 1 - \beta $,
    and 
    \begin{align*}
        \E \[ \|  \nabla^2 F (\x_t) - \wh{H}_t \|^3 | \mathcal{E} \] 
        \!\le\! 
        \(  \frac{ n^2 \tau_2^4}{m_2^3} \!+\! \frac{2 n^2 \sigma_2^4}{m_2^2} \)^{3/4} .
    \end{align*}
    
\end{theorem}

\begin{proof}
    Consistent with Remark \ref{remark:omit_finite_difference_error}, we omit finite-difference gaps throughout our analysis.
    By Jensen and Assumption \ref{assumption:grad-var}, we have  
    \begin{align*} 
        \E \[ \| \nabla F (\x_t) - \g_t \|^{3/2} \]  
        \le 
        \( \E \[ \| \nabla F (\x_t) - \g_t \|^2 \] \)^{3/4} 
        \le 
        \(\frac{\sigma_1^2}{m_1}\)^{3/4} , 
    \end{align*} 
    which proves the first inequality. Next we prove the second inequality.

    Denote by $\wh{H}_t$ the Hessian estimation of $\nabla^2 F(\x_t)$ from Algorithm \ref{alg:est}, and define the empirical Hessian operator $ \wh{\nabla}^2 F (\x_t) := \frac{1}{m_2} \sum_{i=1}^{m_2} {\nabla}^2 f (\x_t; \xi_t^{i,2}) $. Note that
    \begin{align*} 
        & \; \| \nabla^2 F (\x_t) - \wh{\nabla}^2 F (\x_t) \|_F^4 \\ 
        =& \;   
        \< \nabla^2 F (\x_t) - \frac{1}{m_2 } \sum_{i=1}^{m_2} {\nabla}^2 f (\x_t; \xi_t^{i,2}) , \nabla^2 F (\x_t) - \frac{1}{m_2 } \sum_{i=1}^{m_2} {\nabla}^2 f (\x_t; \xi_t^{i,2}) \>^2 \\ 
        =& \;  
        \frac{1}{m_2^4} \sum_{i,j,k,l \in [m_2]} \< \nabla^2 F (\x_t) - {\nabla}^2 f (\x_t; \xi_t^{i,2}) , \nabla^2 F (\x_t) - {\nabla}^2 f (\x_t; \xi_t^{j,2}) \> \\ 
        &\cdot \< \nabla^2 F (\x_t) - {\nabla}^2 f (\x_t; \xi_t^{k,2}) , \nabla^2 F (\x_t) - {\nabla}^2 f (\x_t; \xi_t^{l,2}) \> . 
    \end{align*} 

    Let $ \E_{ \{ \xi_t^{i,2} \}_i } $ be the expectation with respect to $ \{ \xi_t^{i,2} \}_i $. Then taking expectation on both sides of the above inequality gives 
    \begin{align} 
        & \; \E_{ \{ \xi_t^{i,2} \}_i } \| \nabla^2 F (\x_t) - \wh{\nabla}^2 F (\x_t) \|_F^4 \nonumber \\ 
        \le& \;  
        \frac{1}{m_2^4} n^2 \sum_{i=1}^{m_2} \E \| \nabla^2 F (\x_t) - {\nabla}^2 f (\x_t; \xi_t^{i,2}) \|^4 \nonumber \\ 
        & +  \frac{2}{m_2^4} n^2 \!\!\!\! \sum_{i,j \in [m_2], i\neq j} \!\!\!\! \E \| \nabla^2 F (\x_t) - {\nabla}^2 f (\x_t; \xi_t^{i,2}) \|^2 \E \| \nabla^2 F (\x_t) - {\nabla}^2 f (\x_t; \xi_t^{j,2}) \|^2 \nonumber \\ 
        \le& \;  
        \frac{n^2 \tau_2^4}{m_2^3} + \frac{2n^2 \sigma_2^4}{m_2^2} , \label{eq:4-moment-1}
    \end{align} 
    where (\ref{eq:4-moment-1}) uses Assumption \ref{assumption:hess-var}. 
    By Jensen's inequality, it holds that  
    \begin{align}
        & \;\;\| \nabla^2 F (\x_t) - \wh{H}_t \|^4 \nonumber \\ 
        \le&   
         \left\| \nabla^2 F (\x_t) \!-\! \frac{1}{m_2} \sum_{i=1}^{m_2} {\nabla}^2 f (\x_t ; \xi_t^{i,2}) +
        \frac{1}{m_2} 
         \sum_{i=1}^{m_2} \( {\nabla}^2 f (\x_t ; \xi_t^{i,2}) \!-\! \wh{H}_t(\xi_t^{i,2}) \)\right\|^4  \label{eq:4-moment-2}
    \end{align}

    Now, let $\mathcal{E}$ be the following event: 
    \begin{align} 
        \mathcal{E} := \{
            \omega :  \wh{H}_t(\xi_t^{i,2}) = \nabla^2 f (\x_t;\xi_t^{i,2} ),\;  
            \forall i = 1,2,\cdots, m_2, \; \forall t = 0,1,\cdots, T-1.\}
        \label{eq:event}
    \end{align} 

    By Theorem \ref{thm:hess-est}, we know that $ \Pr \( \mathcal{E} \) \ge 1- \beta $. In addition, event $\mathcal{E}$  only depends on $ \{ \u_{i,j} , \v_{i,j} \}_{i,j} $. 
    Therefore, we have 
    \begin{align*}
         \frac{1}{m_2} \sum_{i=1}^{m_2}  \({\nabla}^2 f (\x_t ; \xi_t^{i,2}) - \wh{H}_t(\xi_t^{i,2}) \)=0
    \end{align*}
    under event $\mathcal{E}$. We combine (\ref{eq:4-moment-1}) and (\ref{eq:4-moment-2}) to get
    \begin{align*}
        \E \[ \| \nabla^2 F (\x_t) - \wh{H}_t \|^4 | \mathcal{E} \]
        \le 
        \frac{ n^2 \tau_2^4}{m_2^3} + \frac{2 n^2 \sigma_2^4}{m_2^2}  
    \end{align*} 

    Another use of Jensen gives 
    \begin{align*}
         \E \[ \|  \nabla^2 F (\x_t) - \wh{H}_t \|^3 | \mathcal{E} \] 
        \le 
        \( \E \[ \| \nabla^2 F (\x_t) - \wh{H}_t \|^4 | \mathcal{E} \] \)^{3/4}  , 
    \end{align*} 
    which concludes the proof. \if\proofmark1{\hfill $\square$}\fi
\end{proof}

\begin{remark}
\label{remark:E as whole space}


     The event $\mathcal{E}$ described in equation (\ref{eq:event}) states that the difference between the estimator and the true value is sufficiently small with high probability. 
     Note that $\mathcal{E}$ only depends on $ \{ \u_{i,j} , \v_{i,j} \}_{i,j} $, meaning that $\{ \xi_t^{i,1}, \xi_t^{j,2}\}_{i\in [m_1], j \in [m_2]} $ sampled over $\Theta$ are completely independent of $\mathcal{E}$. 
     Throughout the rest of the paper, we restrict our attention to a sub-probability space, where event $\mathcal{E}$ is always true and the randomness is incurred by sampling over $\Theta$. 
     
\end{remark}

\subsection{Convergence Analysis for Algorithm \ref{alg:cubic_newton_method}} 

Based on the moments result in Section \ref{subsec:grad and hess esti}, we are ready to proceed to the proof of Theorem \ref{thm:epsilon local optima-0}. We first recall a result of \cite{nesterov2006cubic}. 
    


\begin{lemma}[\cite{nesterov2006cubic}] 
    \label{lem:psd-0}
    Let $ \x_{t+1} = \arg\min_{\y} \wh{f}_{t} ( \y ) . $ Then for any $t$, we have 
    \begin{align}  
        & \g_t + \wh{H}_t \( \x_{t+1} - \x_t \) + \frac{\alpha}{2} \| \x_{t+1} - \x_t \| \( \x_{t+1} - \x_t \) = 0 , \label{eq:ns-first-ord} \\ 
        & \wh{H}_t +  \frac{\alpha }{2} \| \x_{t+1} - \x_t \| I \succeq 0 . \label{eq:ns-psd} 
    \end{align} 
\end{lemma}


The proof of Lemma \ref{lem:psd-0} can be viewed as a consequence of Lemma \ref{lem:psd}, which can be found in \cite{nesterov2018lectures}, sec. 4.1.4.1. A proof of Lemma \ref{lem:psd}, using the language of $\Gamma$-convergence, is provided below.

\begin{lemma}[\citep{nesterov2018lectures}, sec. 4.1.4.1]
    \label{lem:psd}
    For any symmetric matrix $H\in\R^{n\times n}$ and $g\in\R^n$, let $v_u(\z):=\< \g, \z \> + \frac{1}{2} \< H \z, \z \> + \frac{M}{6} \| \z \|^3$.
    If $\z^*=\arg\min_{\z\in\R^n} v_u(\z)$ is the global minimizer, then $H+\frac{1}{2}M\|\z^*\|I\succeq 0.$
\end{lemma}

\begin{proof}[Proof of Lemma \ref{lem:psd}]
When $H\succeq 0$, the statement is true. If $H\not\succeq 0$, without loss of generality, we assume that $H=\operatorname{diag}(\lambda_1,\lambda_2,\cdots,\lambda_n)$ and $\lambda_1\geq \lambda_2 \ge \cdots \geq \lambda_m > \lambda_{m+1} = \cdots = \lambda_n $. 
Let $v_u(\z):= \<\g,\z\> + \frac{1}{2}\<H\z,\z\> + \frac{M}{6}\|\z\|^3,\z\in\R^n$ 
and $v_l(r):=-\frac{1}{2}\< (H+\frac{Mr}{2}I)^{-1}\g,\g \> - \frac{M}{12}r^3, r>-\frac{2\lambda_n}{M}$.
We first show that
\begin{equation}
    \min_\z v_u(\z)\geq \sup_{r>-\frac{2\lambda_n }{M}}v_l(r).
\label{eq:minvu>supvl}
\end{equation}
Indeed, 
\begin{align*} 
    \min_\z v_u(\z) 
    &=\min_{\z,\tau\geq0,\|\z\|^2\leq\tau } 
    \<\g,\z\> + \frac{1}{2}\<H\z,\z\> + \frac{M}{6}\|\tau\|^{\frac{3}{2}}\\ 
    &=\min_{\z,\tau\geq0}\sup_{r\geq0}
    \<\g,\z\> + \frac{1}{2}\<H\z,\z\> + \frac{M}{6}\|\tau\|^{\frac{3}{2}} + \frac{M}{4}r(\|\z\|^2-\tau) \\
    &\geq \sup_{r\geq0}\min_{\z,\tau\geq0}
    \<\g,\z\> + \frac{1}{2}\<H\z,\z\> + \frac{M}{6}\|\tau\|^{\frac{3}{2}} + \frac{M}{4}r(\|\z\|^2-\tau) \\
    &= \sup_{r\geq0}\min_{\z}
    \<\g,\z\> + \frac{1}{2}\<(H+\frac{M}{2}rI)\z,\z\> - \frac{M}{12}r^3 \\
    &= \sup_{r>-\frac{2\lambda_n}{M}}\min_{\z}
    \<\g,\z\> + \frac{1}{2}\<(H+\frac{M}{2}rI)\z,\z\> - \frac{M}{12}r^3
    =\sup_{r>-\frac{2\lambda_n}{M}}v_l(r).
\end{align*} 
For $r>-\frac{2\lambda_n }{M}$, $v_l(r)$ is differentiable at any $r>-\frac{2\lambda_n}{M}$ and $v'_l(r)=\frac{M}{4}(\|\z(r)\|^2-r^2)$ where $\z(r):=-(H+\frac{Mr}{2}I)^{-1}\g$. 
In addition, 
\begin{align} 
    v_u (\z(r)) \!-\! v_l (r) 
    = 
    \frac{M}{12} \( r \!+\! 2 \| \z (r) \| \) \( \| \z (r) \| \!-\! r \)^2 
    = 
    \frac{4}{3M} \!\cdot\! \frac{r \!+\! 2 \| \z (r) \|}{ \( r \!+\! \| \z (r) \| \)^2 } \!\cdot\! \( v_l' (r) \)^2, \label{eq:vu-vl} 
\end{align} 

By \eqref{eq:minvu>supvl} and \eqref{eq:vu-vl}, if $\sup_{r>-\frac{2\lambda_n}{M}} v_l(r)$ is attained at some $r^*>-\frac{2\lambda_n}{M}$, then $v_l'(r^*)=0$. This implies that $\min_\z v_u(\z)=\sup_{r>-\frac{2\lambda_n}{M}}v_l(r)$ and $\|\z^*\|=\|\z(r^*)\|=r^*>-\frac{2\lambda_n}{M}$. Therefore $H+\frac{1}{2}M\|\z^*\|I\succeq 0$. 

Otherwise $g_{m+1}=\cdots=g_n=0$. 
And we consider $\g_\delta:=\g+\delta\e_n$ where $\delta>0$ and $\e_n=(0,\cdots,0,1)^\top$.
Let $v_{u,\delta}(\z):=v_u(\z)+\delta z_n=\<\g_\delta,\z\> + \frac{1}{2}\<H\z,\z\> + \frac{M}{6}\|\z\|^3$ and $v_{l,\delta}(r):=-\frac{1}{2}\< (H+\frac{Mr}{2}I)^{-1}\g_\delta,\g_\delta \> - \frac{M}{12}r^3$. 
Then $\lim_{r\to-\lambda_n}v'_{l,\delta}(r)=+\infty$, $\lim_{r\to\infty}v'_{l,\delta}(r)=-\infty$ and by continuity $\sup_{r>-\frac{2\lambda_n}{M}}v_{l,\delta}(r)$ is achieved by some $r_\delta^*>-\frac{2\lambda_n}{M}$. 
As a consequence, the minimizer of $v_{u,\delta}(\z)$ is $\z_\delta^*=(H+\frac{M}{2}r_\delta^*I)^{-1}\g_{\delta}$ and $\|\z_\delta^*\|=r_\delta^*$.
Consider $\delta_k=\frac{1}{k}$, then $\{v_{u,\delta_k}\}_{k=1}^{\infty}$ is equi-coercive in $\R^n$ and $\Gamma$-converge to $v_u$, $\{v_{l,\delta_k}\}_{k=1}^{\infty}$ is equi-coercive in $\{r:r>-\frac{2\lambda_n}{M}\}$ and $\Gamma$-converge to $v_l$. 
By fundamental theorem of $\Gamma$-convergence, $\z_0^*:=\lim_{k\to\infty}\z_{\delta_k}^*$ is the optimizer of $v_u$ and $\lim_{k\to\infty}r_{\delta_k}^*$ is the optimizer of $v_l$. By continuity of norm and since $-\frac{2\lambda_n}{M}$ is the unique minimizer of $v_l$, we have $\|\z_0^*\|=\lim_{k\to\infty}r_{\delta_k}^*=-\frac{2\lambda_n}{M}$ and therefore $H+\frac{1}{2}M\|\z^*\|I\succeq 0$. \if\proofmark1{\hfill $\square$}\fi 
\end{proof}

To draw the final conclusion, we need the following two additional lemmas.

\begin{lemma} 
\label{lem:diff_betw_xt&xt+1}
    Instate Assumptions \ref{as:H lip cts }, \ref{assumption:low-rank}, \ref{assumption:grad-var} and \ref{assumption:hess-var}. Let $\alpha \ge L_2$. 
    Then it holds that 
    \begin{align*}
         \frac{\alpha}{36} \E \[ \| \x_{t+1} - \x_t \|^3  \] 
        \le& \;
        \E \[ F (\x_t)  \] - \E \[ F (\x_{t+1})  \] 
         \\
       & \;
        + 
       \sqrt{\frac{16}{3\alpha}}\(\frac{\sigma_1^2}{m_1}\)^{\frac{3}{4}}
       + \frac{1}{6} \(\frac{12}{\alpha} \)^2 \(  \frac{ n^2 \tau_2^4}{m_2^3} + \frac{2 n^2 \sigma_2^4}{m_2^2} \)^{\frac{3}{4}}, 
    \end{align*}
    where the expectation is taken in the sub-probability space defined by event $\mathcal{E}$, as pointed out in Remark \ref{remark:E as whole space}.
    
\end{lemma}

\begin{proof}
    By (\ref{eq:def-f-hat}) and (\ref{eq:ns-first-ord}), we have 
    \begin{align*} 
        \wh{f}_{t} (\x_{t+1} ) 
        = 
        - \frac{1}{2} \< \x_{t+1} - \x_t, \wh{H}_t \( \x_{t+1} - \x_t \) \> - \frac{\alpha}{3} \| \x_{t+1} - \x_t \|^3 
        \le 
        - \frac{\alpha }{12} \| \x_{t+1} - \x_t \|^3 . 
    \end{align*} 

    By smoothness, for $\alpha \geq L_2$, we have 
    \begin{align*}
        F (\x_{t+1}) 
        \le& \;  
        F (\x_t ) + \wh{f}_{t} \( \x_{t+1} \) + \< \nabla F (\x_t) - \g_t , \x_{t+1} - \x_t \> \\ 
        &+ \frac{1}{2} \< \x_{t+1} - \x_t, \( \nabla^2 F (\x_t)-\wh{H}_t \) \( \x_{t+1} - \x_t \) \> \\ 
        \le& \; 
        F (\x_t) - \frac{\alpha }{12} \| \x_{t+1} - \x_t \|^3 + \< {\nabla} F (\x_t) - \g_t , \x_{t+1} - \x_t \> \\
        &+ \frac{1}{2} \< \x_{t+1} - \x_t, \( \nabla^2 F (\x_t)-\wh{H}_t \) \( \x_{t+1} - \x_t \) \> \\ 
        \le& \; 
        F (\x_t) - \frac{\alpha }{12} \| \x_{t+1} - \x_t \|^3 + \sqrt{\frac{16}{3\alpha}} \| {\nabla} F (\x_t) - \g_t \|^{3/2} + \frac{\alpha}{36} \| \x_{t+1} - \x_t \|^3 \\
        &+ \frac{1}{2} \cdot \frac{1}{3} \cdot \( \frac{12}{\alpha} \)^2 \| \nabla^2 F (\x_t) - \wh{H}_t \|^3 + \frac{1}{2} \cdot \frac{2}{3} \cdot \frac{\alpha}{12} \|\x_{t+1} - \x_t \|^3 , 
    \end{align*}
    where the last line uses the inequalities of Cauchy-Schwarz and Young.

    Note that we treat the event $\mathcal{E}$ as the whole probability space, as the previous statement in Remark \ref{remark:E as whole space}.  We take expectation on both sides of the above inequality, and apply Theorem \ref{thm:moment-bound} to get   
    \begin{align*} 
         \frac{\alpha}{36} \E \[ \| \x_{t+1} - \x_t \|^3  \]  
        \le& \;  
        \E \[ F (\x_t)  \] - \E \[ F (\x_{t+1})  \] 
        \\
         & \;
        + 
        \sqrt{\frac{16}{3\alpha}} \E \[ \| {\nabla} F (\x_t) \!-\! \g_t \|^{3/2}  \] 
        \!+\! \frac{1}{6} \( \frac{12}{\alpha} \)^2 \!\E \[ \| \nabla^2 F (\x_t) - \wh{H}_t \|^3  \] 
        \\ 
        \le& \; 
        \E \[ F (\x_t)  \] - \E \[ F (\x_{t+1})  \] 
        \\
        & \;
        + 
        \sqrt{\frac{16}{3\alpha}} \(\frac{\sigma_1^2}{m_1}\)^{3/4}  
         + 
         \frac{1}{6} \( \frac{12}{\alpha} \)^2 \(  \frac{ n^2 \tau_2^4}{m_2^3} + \frac{2 n^2 \sigma_2^4}{m_2^2} \)^{3/4},  
    \end{align*} 
    which concludes the proof. \if\proofmark1{\hfill $\square$} \fi 
\end{proof}

\begin{lemma}
\label{lem:eigenvalue_part}
    Instate Assumptions \ref{as:H lip cts }, \ref{assumption:low-rank}, \ref{assumption:grad-var} and \ref{assumption:hess-var}. Let $\alpha \ge L_2$. For any $\kappa \in (0,1)$, let $m_1=\(\frac{\sigma_1}{\kappa}\)^2$ and  $m_2=\frac{n\sqrt{\tau_2^4+2\sigma_2^4}}{2(L_2+\alpha)\kappa}$, then we have
    \begin{align*}
        \sqrt{\E\[\|\x_{t+1}-\x_t\|^2 \]} \geq 
        \max 
        \begin{cases}
            \sqrt{\frac{\E\[\|\nabla F(\x_{t+1})\|\]-2\kappa}{L_2+\alpha}} \\
            -\frac{2}{\alpha+2L_2}\[ \sqrt{2\kappa\(L_2+\alpha\)}+\E\[ \lambda_{\min}\(\nabla^2 F\(\x_{t+1}\)\)\] \] 
        \end{cases}. 
    \end{align*}
\end{lemma}
\begin{proof}
    By (\ref{eq:ns-first-ord}), we have 
    \begin{align*}
        \nabla F\(\x_{t+1}\)
        =&\;
        \nabla F\(\x_{t+1}\)-\nabla F\(\x_{t}\)+\nabla F\(\x_{t}\)\\
        =&\;
        \nabla F\(\x_{t}\)-\g_t - \wh{H}_t \( \x_{t+1} - \x_t \) - \frac{\alpha}{2} \| \x_{t+1} - \x_t \| \( \x_{t+1} - \x_t \) \\
        &\;
        +\!\nabla F\(\x_{t+1}\)\!-\!\nabla F\(\x_{t}\)\!+\!\nabla^2F\(\x_t\)\(\x_{t+1}\!-\!\x_t\)\!-\!\nabla^2F\(\x_t\)\(\x_{t+1}\!-\!\x_t\).
    \end{align*}
    Rearranging terms gives
    \begin{align*}
         &\|\nabla F(\x_{t+1})\| 
         \\
        \leq&
        \|\nabla F\(\x_{t+1}\)\!-\!\nabla F\(\x_{t}\)\!-\!\nabla^2F\(\x_t\)\(\x_{t+1}\!-\!\x_t\)\|\!+\!\|\nabla F\(\x_{t}\)\!-\!\g_t\| \\
        &
        +\|\(\nabla^2F\(\x_t\)\!-\!\wh{H}_t\)\(\x_{t+1}\!-\!\x_t\)\|\!+\!\frac{\alpha}{2} \| \x_{t+1} \!-\! \x_t \|^2  \\
        \leq &
        \frac{L_2\!+\!\alpha}{2}\| \x_{t+1} \!-\! \x_t \|^2  \!+\!\|\nabla F\(\x_{t}\)\!-\!\g_t\|\!+\!\|\nabla^2 F\(\x_t\)\!-\!\wh{H}_t\|\!\cdot\! \|\x_{t+1}\!-\!\x_t\| \\
        = &
        \frac{L_2\!+\!\alpha}{2}\| \x_{t+1} \!-\! \x_t \|^2  \!+\!\|\nabla F\(\x_{t}\)\!-\!\g_t\| 
        \!+\!\frac{\|\nabla^2F\(\x_t\)\!-\!\wh{H}_t\|}{\sqrt{L_2\!+\!\alpha}}\!\cdot\!\sqrt{L_2 \!+\!\alpha}\|\x_{t+1}\!-\!\x_t\|  \\
        \leq &
        (L_2\!+\!\alpha)\| \x_{t+1} \!-\! \x_t \|^2\!+\!\|\nabla F\(\x_{t}\)\!-\!\g_t\|\!+\!\frac{1}{2(L_2\!+\!\alpha)}\|\nabla^2F\(\x_t\)\!-\!\wh{H}_t\|^2.
    \end{align*}
    Hence, we have
    \begin{align*}
        \| \x_{t+1} \!-\! \x_t \|^2 
        \!\geq \! 
        \frac{1}{L_2\!+\!\alpha}\|\nabla \! F(\x_{t+1}\!)\|\!-\!\frac{1}{L_2\!+\!\alpha}\|\nabla \!F\!\(\x_{t}\)\!-\!\g_t\|  
        \!-\!\frac{1}{2(L_2\!+\!\alpha)^2}\|\nabla^2 \!F\!\(\x_t\)\!-\!\wh{H}_t\|^2.
    \end{align*}
    By  Theorem \ref{thm:moment-bound}, for any $\kappa \in (0,1)$, picking $m_1=\(\frac{\sigma_1}{\kappa}\)^2$ and  $m_2=\frac{n\sqrt{\tau_2^4+2\sigma_2^4}}{2(L_2+\alpha)\kappa}$  gives 
    \begin{align}
        &\; \E \|\nabla F\(\x_{t}\)-\g_t\| \leq  \kappa,   \quad \text{and} \nonumber \\ 
        &\E \[\|\nabla^2F\(\x_t\)-\wh{H}_t\|^2 \] \leq \(\E \[\|\nabla^2F\(\x_t\)-\wh{H}_t\|^3 \]\)^{\frac{2}{3}}\leq 2(L_2 +\alpha)\kappa.
        \label{eq:expec_of_gap_betw_hess}
    \end{align}
    And then we can get
    \begin{align}
        \E \[\| \x_{t+1} - \x_t \|^2 \]\geq 
        \frac{1}{L_2 +\alpha}\E\[\|\nabla F(\x_{t+1})\|\]-\frac{2\kappa}{L_2 +\alpha}.
        \label{eq:eigenvalue_lem_part1}
    \end{align}
    For another term in the conclusion, by Assumption \ref{as:H lip cts } and  (\ref{eq:ns-psd}), we have 
    \begin{align*}
        \nabla^2 F\(\x_{t+1}\) 
        \succeq & \;
        \nabla^2 F\(\x_{t}\)- L_2 \|\x_{t+1} - \x_t\|I_n \\
        = &\;
        \nabla^2 F\(\x_{t}\)-\wh{H}_t+\wh{H}_t-L_2\|\x_{t+1} - \x_t\|I_n \\
        \succeq & \;
        \nabla^2 F\(\x_{t}\)-\wh{H}_t-\frac{\alpha+2L_2}{2}\|\x_{t+1} - \x_t\|I_n.
    \end{align*}
    Analysing the minimum eigenvalue and taking expectation on the both sides, it holds that 
    \begin{align}
        & \; \sqrt{\E \[\| \x_{t+1} - \x_t \|^2 \]} \nonumber  
        \geq \;
        \E \[\| \x_{t+1} - \x_t \| \] \nonumber\\
        \geq &\;
        \frac{2}{\alpha+2L_2}\[ \E \[\lambda_{\min}\(\nabla^2 F\(\x_{t}\)-\wh{H}_t\)\]-
        \E\[ \lambda_{\min}\(\nabla^2 F\(\x_{t+1}\)\)  \] \] \nonumber\\
        \geq &\;
        -\frac{2}{\alpha+2L_2}\[\sqrt{2\kappa(L_2+\alpha)}+\E \[\lambda_{\min}\(\nabla^2 F\(\x_{t+1}\)\) \]\],
        \label{eq:eigenvalue_lem_part2}
    \end{align}
    where the last inequality comes from (\ref{eq:expec_of_gap_betw_hess}). Combining (\ref{eq:eigenvalue_lem_part1}) and (\ref{eq:eigenvalue_lem_part2}), we come to the conclusion. \if\proofmark1{\hfill $\square$}\fi 
\end{proof}
\begin{proof}[Proof of Theorem \ref{thm:epsilon local optima-0}]
    By Lemma \ref{lem:diff_betw_xt&xt+1}, we know that 
    \begin{align*} 
          \frac{1}{T} \!\sum_{t=0}^{T-1} \E\! \[ \| \x_{t+1} \!-\! \x_t \|^3  \] 
        \!\leq 
        \frac{36}{\alpha}\!\cdot\! \frac{ F(\x_0) \!-\! F^*}{T}\!+\!\frac{48\sqrt{3}}{\alpha^{3/2}} \!\cdot\! \(\frac{\sigma_1^2}{m_1}\)^{\frac{3}{4}} 
        +\!
        \frac{864}{\alpha^3}\!\!\(\!  \frac{ n^2 \tau_2^4}{m_2^3} \!+\! \frac{2 n^2 \sigma_2^4}{m_2^2}  \!\)^{\frac{3}{4}}. 
    \end{align*}
    By Jensen's inequality, law of total expectation and picking the parameters in Theorem \ref{thm:epsilon local optima-0}, it holds that
    \begin{align*}
        & \; \E \[ \| \x_{R+1} - \x_R \|^2  \] 
        \leq 
       \( \E \[ \| \x_{R+1} - \x_R \|^3  \] \)^{\frac{2}{3}} \\ 
       = &\;
       \( \E \(\E_R \[ \| \x_{R+1} - \x_R \|^3  \] \) \)^{\frac{2}{3}} 
       = 
       \(\frac{1}{T} \sum_{t=0}^{T-1} \E \[ \| \x_{t+1} - \x_t \|^3  \]\)^{\frac{2}{3}} \\ 
       \leq &\;
       \frac{1}{800}\[ \(1+48\sqrt{3}+864\times8\)\frac{\eta^\frac{3}{2}}{L_2^\frac{3}{2}}\]^\frac{2}{3} \leq \; \frac{37}{80}\cdot \frac{\eta}{L_2}.
    \end{align*}
    And then by Lemma \ref{lem:eigenvalue_part} with $\kappa=\eta/800$, we can draw the conclusion. 
    \if\proofmark1{\hfill $\square$}\fi 
\end{proof}


    


%% file: tex/experiment.tex
\section{Experiments}
\label{sec:exp}

So far, we have introduced our main algorithm and proved its convergence. In this section, we will conduct the numerical experiments, which demonstrate the efficiency of the zeroth-order stochastic cubic Newton method with low-rank Hessian estimation (Algorithm \ref{alg:cubic_newton_method}). 

\subsection{Zeroth-order Cubic Newton Method with Low-rank Hessian Estimation}
Based on the low-rank   Hessian estimation techniques, we will perform numerical experiments with zeroth-order stochastic cubic Newton method, solving logistic regression optimization problem on the iris dataset \citep{iris_53}. 

The objective function is the logistic loss defined as follows:
\begin{align*}
    \min_{\x \in \R^n} \frac{1}{N}\sum_{\xi=1}^N f (\x;\xi) ,
    \qquad 
    f (\x;\xi) = \log\left(1+e^{-y_\xi \z_{\xi}^\top \x }\right),
\end{align*}
where 
$ \( \z_{\xi}, y_\xi \) $ ($y_{\xi} \in \{-1,+1\}$) is the feature and label of the $\xi$-th training data.

One can easily verify that $rank \left( \nabla^2 f (\x;\xi)\right)=1$ for any $\xi$ and $\x$. To demonstrate the effectiveness and efficiency of the proposed methods, we conduct a strict implementation of Algorithm \ref{alg:cubic_newton_method}, and compare it with the a standard zeroth-order stochastic gradient descent algorithm (ZO-SGD). The comparison between these two methods can be found in Figure \ref{fig:exp-sample_complexity}. The parameter configuration for Algorithm \ref{alg:cubic_newton_method} (resp. ZO-SGD) can be found in Table \ref{tab:parameter of cubic newton method} (resp. Table \ref{tab:parameter of ZO-SGD}). 

To ensure a fair comparison, we perform numerical experiments for ZO-SGD under the same problem settings as Algorithm \ref{alg:cubic_newton_method}. Both experiments have a batch size of 5 and start at the same initial point $\x_0$, which is randomly sample from the standard Gaussian distribution; See Tables \ref{tab:parameter of ZO-SGD} and \ref{tab:parameter of cubic newton method} for details. We repeat the experiments for 10 times, and summarize the results in Figure \ref{fig:exp-sample_complexity}. 

\begin{table}[h]
    \centering
    \begin{tabular}{cc}
         \toprule 
          \makecell{parameter} & \makecell{ZO-SGD} \\ \hline 
         batch size for gradient & 5 \\ 
         finite difference granularity & 0.001 \\ 
         step size & $\gamma=$ 1, 0.1, 0.001 \\ 
        \bottomrule 
    \end{tabular}
    \caption{parameters of ZO-SGD} 
    \label{tab:parameter of ZO-SGD} 
\end{table} 

\begin{table}[h] 
    \centering 
    \begin{tabular}{ccc} 
         \toprule 
        \makecell{parameter} &\makecell{explanation} & \makecell{Algorithm 3} \\ \hline 
         $m_1$& batch size for gradient &5 \\ 
         $m_2$ &  batch size for Hessian &5 \\ 
         $M$ & number of measurements &8 \\ 
         $\delta$ & finite difference granularity &0.001 \\ 
         $\alpha$ & regularizer for Newton step  &1 \\ 
        \bottomrule 
    \end{tabular}
    \caption{parameters of Algorithm 3} 
    \label{tab:parameter of cubic newton method} 
\end{table}

\begin{figure}[h!]
\centering
\includegraphics[width=0.8\linewidth]{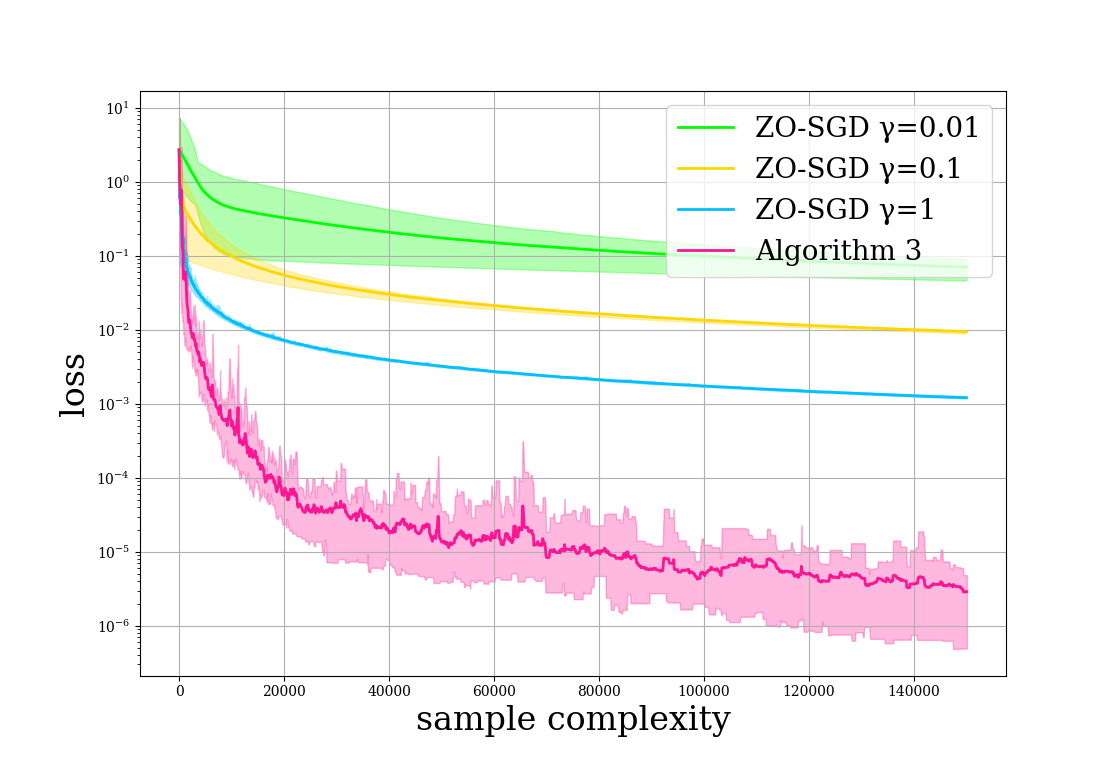}
		\caption{Training losses versus sample complexity for Algorithm \ref{alg:cubic_newton_method} and ZO-SGD. The $x$-axis indicates number of function evaluations used. The solid lines are averaged over 10 runs. The shaded area above and below the solid lines shows the range (the maximum and minimum) of the 10 runs. For this experiment, ``Iris Versicolour'' and ``Iris Virginica'' are combined into a single class. } 
		\label{fig:exp-sample_complexity}

  
\end{figure}

\section{Conclusion} 
\label{sec:conc}

This paper investigates the stochastic optimization problem, particularly emphasizing the common scenario where the Hessian of the objective function over a data batch is low-rank. We introduce a stochastic zeroth-order cubic Newton method that demonstrates enhanced sample complexity compared to previous state-of-the-art approaches. This improvement is made possible through a novel Hessian estimation technique that capitalizes on the benefits of matrix completion while mitigating its limitations. The effectiveness of our method is supported by elementary numerical studies.

%% file: tex/appendix.tex
\section{Hessian of Neural Networks}

Below, we analyze linear neural networks, where $\sigma_i=id$ for all $i$. This simplification is sufficient to illustrate the Hessian structure for ReLU activation, as ReLU either outputs zero or behaves like the identity function. Following convention, we often work with matrix-matrix derivatives by vectorizing row-wise in the numerator (Jacobian) layout \citep{singh2021analytic,ormaniec2024does}.
    By chain rule, we obtain the first-order differential matrix 
    \begin{itemize}
        \item[] \hspace*{-2em} $\frac{\partial f(W;\xi)}{\partial W_i} =2\(W_L W_{L-1}\cdots W_{i+1}\)^\top\(W_L  \cdots  W_2 W_1 \z_\xi - y_\xi \) \(W_{i-1}\cdots W_1\z_\xi\)^\top.$ 
    \end{itemize}
   Then following the calculations in \cite{singh2021analytic}, we know the $i$-th  diagonal block of Hessian is  
   \begin{itemize}
        \item[] $ \nabla^2_{W_i} f (W; \xi) :=\frac{\partial^2 f(W;\xi)}{\partial W_i \partial W_i} $ $=$ $ \frac{\partial \, vec_r(\partial f(W;\xi)/\partial W_i)}{\partial \, vec_r(W_i)^\top}
        $ \\
        $=$ $2\frac{\partial \, vec_r\(\(W_L \cdots W_{i+1}\)^\top\(W_L \cdots W_{i+1}\)\cdot W_i \cdot \(W_{i-1}\cdots W_1\z_\xi\) \(W_{i-1}\cdots W_1\z_\xi\)^\top \)}{\partial \, vec_r(W_i)^\top}
        $ \\
        $=$ $ 2 \frac{\partial \, \[\(W_L \cdots W_{i+1}\)^\top\(W_L \cdots W_{i+1}\) \otimes \(W_{i-1}\cdots W_1\z_\xi\) \(W_{i-1}\cdots W_1\z_\xi\)^\top\]vec_r(W_i)}{\partial \, vec_r(W_i)^\top} 
        $ \\
        $=$ $2\(W_L \cdots W_{i+1}\)^\top\(W_L \cdots W_{i+1}\) \otimes \(W_{i-1}\cdots W_1\z_\xi\) \(W_{i-1}\cdots W_1\z_\xi\)^\top, $
    \end{itemize}
    where $vec_r(\cdot)$ is the operation of row-stacking, $\otimes $ is the Kronecker product, and $\nabla^2_{W_i}f(\cdot;\xi)$ denotes the Hessian matrix of $f(\cdot;\xi)$ with respect to $W_i$. 
    We observe that $rank\(\(W_L \cdots W_{i+1}\)^\top\(W_L \cdots W_{i+1}\)\)\leq\min\{d_{L},d_{L-1},\cdots,d_{i}\}$ and $rank\(\(W_{i-1}\cdots W_1\z_\xi\) \(W_{i-1}\cdots W_1\z_\xi\)^\top\) \leq 1$. Therefore, the rank  of Hessian with respect to $W_i$ satisfies  $rank(\nabla^2_{W_i}f(W; \xi))\leq\min\{d_{L},d_{L-1},\cdots,d_{i}\}$, and can be very small. Given that the size of $\nabla^2_{W_i}f(W; \xi)$ is $d_{i}d_{i-1} \times d_{i}d_{i-1}$, it clearly admits a low-rank structure. 

    Additionally, during training, the rank of the weight matrix may decrease \citep{sagun2016eigenvalues, singh2021analytic}, further reducing the rank of the Hessian matrix.

\begin{remark}
    When considering the Hessian in the context of DNN training, it is worth reiterating two key points: 
    \begin{enumerate}
    \item 
    It is sufficient to compute only the "partial" Hessian with respect to the parameter matrix $ W_l $ at a specific layer depth. This is because, in practice, DNNs are trained using block coordinate descent methods, such as back-propagation \citep{Rumelhart1986} or more recent zeroth-order forward methods for fine-tuning \citep{malladi2023finetuning,chen2023deepzero,liu2024sparse}.
    \item 
    Our derivation using linear activation is adequate to demonstrate the behavior of the ReLU activation, which is the most widely used activation function. This is due to the fact that ReLU either outputs zero or behaves like the identity function.
    \end{enumerate} 
\end{remark}

\clearpage
\section{Table of Notations}
In this section, we have compiled a comprehensive table of notations within this paper to facilitate quick reference and enhance readability.
\begin{table}[h!]
\centering
\caption{List of notations.}
\begin{tabular}{c|c}
\hline
Notation                    & Explanation                                                        \\ \hline
$n$                         & dimension of the stochastic optimization problem                   \\
$r$                         & upper bound of the rank of Hessian                                 \\
$F(\x)$                     & the objective function                                             \\
$F^*$                       & the minimum of $F$                                                 \\
$\xi$                       & index for a (batch of) training data                                         \\
$f(\x;\xi)$                 & function evaluation corresponding to $\xi$                         \\
$\P_i$                      & matrix measurement operation                                       \\
$\S$                        & the sampler constructed with $\P_i$                                \\
$T$                         & the total time horizon                                             \\
$m_1$                       & the number of function evaluations used for gradient               \\
$m_2$                       & the number of function evaluations used for Hessian                \\
$\u_i,\v_i$                     & \makecell{measurement vectors uniformly distributed over the unit sphere} \\
$L_1$                       & gradient Lipschitz constant                                        \\
$L_2$                       & Hessian Lipschitz constant                                         \\
$\delta$                    & finite-difference granularity                                      \\
$M$                         & the number of measurements                                         \\
$H$                         & a random symmetric matrix determined by $\{\u_i,\v_i\}_i$                                         \\
$H_r$                       & the best rank-$r$ approximation of the real Hessian                     \\
$\wh{H}$                    & the solution of the convex program (\ref{eq:goal})             \\
$\g_t$                      & the gradient estimator at time $t$                                 \\
$\wh{H}_t$                  & the Hessian estimator at time $t$                                  \\
$\alpha$                    & regularizer for Newton step                                        \\
$\epsilon$                  & the gap between $H_r$ and the real Hessian                         \\
$\beta,\eta,\kappa$         & arbitrarily small quantity                                         \\
$\mathcal{A}$               & a subspace of $\R^{n\times n}$                  \\
$\P_\mathcal{A}$            & a projection operation onto $\mathcal{A} $                         \\
$\sigma_1,\sigma_2, \tau_2$ & the upper bound of some central moments                           
\end{tabular}
\end{table}